\documentclass[11pt,times,letter]{article}

\usepackage{amsmath, amsfonts,amsthm,amssymb,bm}
\usepackage{dsfont}

\usepackage[english]{babel}
\usepackage{float}

\usepackage[normalem]{ulem}

\usepackage{subfigure}
\usepackage{latexsym,amssymb,amsfonts,graphicx}
\usepackage{amsmath,dsfont}
\usepackage{verbatim}
\usepackage{mathrsfs}
\usepackage{bm}
\usepackage{color}
\usepackage{epsfig}
\usepackage{epstopdf}
\usepackage[title]{appendix}
\usepackage{url}
\usepackage{bm}
\usepackage{natbib}

\usepackage[colorlinks,linkcolor=blue,citecolor=blue,anchorcolor=blue]{hyperref}

\newcommand{\Ex}{ \mathbb{E} }

\renewcommand{\P}{\mathbb{P}}

\def\esssup_#1{\underset{#1}{\mathrm{ess\,sup\, }}}
\def\essinf_#1{\underset{#1}{\mathrm{ess\,inf\, }}}
\def\argmax_#1{\underset{#1}{\mathrm{arg\,max\, }}}
\def\argmin_#1{\underset{#1}{\mathrm{arg\,min\, }}}

\newcommand{\Fx}{\mathbb{F} }

\newcommand{\F}{\mathcal{F}}

\newcommand{\R}{\mathds{R}}

\newtheorem{theorem}{Theorem}[section]

\numberwithin{equation}{section}

\newtheorem{proposition}[theorem]{Proposition}
\newtheorem{remark}[theorem]{Remark}
\newtheorem{lemma}[theorem]{Lemma}

\allowdisplaybreaks

\definecolor{Red}{rgb}{1.00, 0.00, 0.00}

\definecolor{DRed}{rgb}{0.5, 0.00, 0.00}

\definecolor{Blue}{rgb}{0.00, 0.00, 1.00}

\definecolor{Green}{rgb}{0.0, 0.4, 0.0}

\title{A Decomposition-Homogenization Method for Robin Boundary Problems on the Nonnegative Orthant}

\author{
Lijun Bo \thanks{Email: lijunbo@ustc.edu.cn,  School of Mathematics and Statistics, Xidian University, Xi'an, 710126, China.}
\and
Yijie Huang \thanks{Email: huang1@mail.ustc.edu.cn, School of Mathematical Sciences, University of Science and Technology of China, Hefei, 230026, China.}
\and
Xiang Yu \thanks{Email: xiang.yu@polyu.edu.hk, Department of Applied Mathematics, The Hong Kong Polytechnic University, Hung Hom, Kowloon, Hong Kong.}
}

\date{\vspace{-0.5in}}

\begin{document}
\maketitle

\begin{abstract}

This paper studies the existence and uniqueness of a classical solution to a type of Robin boundary problems on the nonnegative orthant. We propose a new decomposition-homogenization method for the Robin boundary problem based on probabilistic representations, which leads to two auxiliary Robin boundary problems admitting some simplified probabilistic representations. The auxiliary probabilistic representations allow us to establish the existence of a unique classical solution to the original Robin boundary problem using some stochastic flow analysis.

\vspace{0.1in}


\end{abstract}

\section{Introduction}\label{sec:intro}

The Feynman-Kac formula for the linear second-order PDE with Neumann boundary conditions on a bounded domain with smooth boundary has been well studied where the probabilistic representation involves a reflecting Brownian motion and its local time on the boundary.  As a pioneer study, \cite{Hsu1985} established the probabilistic solution for the linear elliptic PDE in $\R^d$ with a Neumann boundary condition on a bounded domain with smooth boundary. This probabilistic solution is understood as a weak solution to the PDE problem and is proved to be a classical solution in \cite{Brosamler1976} and \cite{Hsu1985} if the solution satisfies some smoothness conditions. By unifying and generalizing the results in \cite{Brosamler1976}, \cite{Hsu1985}  and \cite{masong1988} provided a probabilistic approach to the Neumann problem of Schr\"odinger equations without the assumption of the finiteness of the gauge. On a bounded domain in $\R^d$ with smooth boundary, \cite{zhang1990} gave a martingale formulation of the Neumann problem of Schr\"odinger operator with measure potential, and provided an analytic characterization of the martingale formulation by using the Dirichlet forms in terms of reflecting Brownian motion (RBM) and its boundary local time. \cite{Hu1993} showed that a classical solution to a class of the Neumann problems of quasilinear elliptic PDEs has a probabilistic representation on a bounded domain in $\R^d$ with a $C^3$-boundary by studying the backward stochastic differential equations (BSDEs) in the infinite horizon case. \cite{wongetal2022} established the existence and uniqueness of the weak solution to a Neumann problem of an elliptic PDE with a singular divergence term on a bounded domain in $\R^d$ with smooth boundary. Therein, a probabilistic approach is applied by studying the BSDE associated with the PDE. The probabilistic solution of the Neumann problem can make numerical implementations and approximations efficiently.  \cite{MilTre2002} considered a class of layer methods for solving the Neumann problem for semilinear parabolic equations based on probabilistic solutions. \cite{zhouetal2017} proposed numerical methods for computing the boundary local time of RBM on a bounded domain in $\R^3$, and its application in finding accurate approximation of the local time and discretization of the probabilistic representation of the Neumann problem using the computed local time. Some research extensions have been carried out from the Neumann boundary problem to the Robin (also known as the third type) boundary problem that specifies the boundary conditions on the linear combination of the solution and its derivative.  \cite{Papanicolaou1990} established a probabilistic solution for the the Robin boundary problem on a bounded domain with $C^3$ boundary. This probabilistic solution can be proved to be unique, continuous and equivalent to the weak analytic solution.  \cite{LeiShatre23} studied elliptic PDEs with Robin boundary conditions and discussed their connection to reflected diffusion processes via the Feynman-Kac formula.

The previous studies on probabilistic solutions to Neumann boundary problems and Robin boundary problems predominantly assume the smoothness of the boundary of the (bounded) domain. The local time acted on the smooth boundary as an additive functional with respect to the surface measure admits an explicit form in terms of transition density of the reflected state process (\citealt{wongetal2022}). However, this convenient form is not available when the state space of the reflected process has corners. In this paper, we consider the following Robin boundary problem of parabolic type defined on the nonnegative orthant for the spatial variables:
\begin{align}\label{eq:PDE-multidim}
\begin{cases}
\displaystyle (\partial_t+\mathcal{L})u(t,\mathbf{x}) = \rho u(t,\mathbf{x}),\quad \text{on}~(t,\mathbf{x})\in[0,T)\times\R_+^d,\\[0.6em]
\displaystyle u(T,\mathbf{x})=g(\mathbf{x}),\quad \forall \mathbf{x}\in\overline{\R}_+^{d},\\[0.6em]
\displaystyle \partial_{x_i}u(t,(\mathbf{x}^{-i},0))+c_i u(t,(\mathbf{x}^{-i},0))=f_i(t,\mathbf{x}^{-i}),\quad \forall (t,\mathbf{x}^{-i})\in[0,T)\times\overline{\R}_+^{d-1},\\[0.4em]
\displaystyle \qquad\qquad\qquad\qquad\qquad\qquad\qquad\qquad\qquad\qquad\quad~i=1,\ldots,d,
\end{cases}
\end{align}
where $d\geq2$, $\rho\geq0$ is the killing rate, $T>0$ is the terminal time, $(c_1,\ldots,c_d)\in\R^d$,  $(\mathbf{x}^{-i},0):=(x_1,x_2,...,x_{i-1},0,,x_{i+1},...,x_d)$, $\R_+:=(0,\infty)$ and $\overline{\R}_+=[0,\infty)$. The second-order differential operator $\mathcal{L}$ acting on $\phi\in C^2(\R_+^d)$ is given by
\begin{align}\label{eq:Lmulti}
\mathcal{L}\phi(\mathbf{x}) &:= \frac{1}{2}\sum_{i,j=1}^d \left(\sum_{k=1}^m\sigma_{ik}\sigma_{jk}\right)\partial_{x_i x_j}^2\phi(\mathbf{x})+\sum_{i=1}^d\mu_i \partial_{x_i}\phi(\mathbf{x})
\end{align}
with $(\sum_{k=1}^m\sigma_{ik}\sigma_{jk})_{i,j=1}^d$ being a positive definite matrix and $\mu_i\in\R$ for $i=1,\ldots,d$.

The following assumption on the terminal condition and boundary functions is imposed throughout the paper:
\begin{itemize}
\item[{\bf(A)}]  $g\in C^2(\overline{\R}_+^d)$ and $f_i\in C^{1,2}([0,T]\times\overline{\R}_+^{d-1})$ for all $i=1,\ldots,d$. There exists some constant $C>0$ and $q\geq 1$  such that $|h(\mathbf{x})|\leq C(1+|\mathbf{x}|^q)$, $|\partial_{x_j}h(\mathbf{x})|\leq C(1+|\mathbf{x}|^q)$ and $|\partial_{x_jx_k}^2h(\mathbf{x})|\leq C(1+|\mathbf{x}|^q)$ for $\mathbf{x}\in\overline{\R}_+^d$, $h\in\{f_i(t),g\}$ and $j,k=1,\ldots,d$.
\end{itemize}
 For the domain with non-smooth boundary,  \cite{DupuisIshii1991} handled the well-posedness of a class of fully nonlinear elliptic PDEs with the oblique derivative conditions in the viscosity sense.  \cite{IsKu2022} established a comparison theorem for viscosity sub- and supersolutions of the nonlinear Neumann problem of elliptic PDEs on a two-dimensional quadrant.  With Assumption {\bf(A)}, a classical solution $u$ of Robin boundary problem \eqref{eq:PDE-multidim} (i.e., $u\in C^{1,2}([0,T)\times\R_+^d)\cap C([0,T]\times\overline{\R}_+^d)$ and it solves Eq.~\eqref{eq:PDE-multidim}) can be expressed by the Feynman-Kac formula that
 (\citealt{Freidlin1985,Papanicolaou1990}):
\begin{align}\label{eq:PRprimalNeumann}
u(t,\mathbf{x}) &= -\sum_{i=1}^d\Ex\left[\int_t^{T}e^{-\rho (s-t)+\sum_{k=1}^d c_k L_s^{k,t}}f_i(s,\mathbf{X}_s^{-i,t,\mathbf{x}})dL_s^{i,t}\right]\nonumber\\
&\quad+\Ex\left[e^{-\rho(T-t)+\sum_{k=1}^d c_k L_T^{k,t}} g(\mathbf{X}_T^{t,\mathbf{x}})\right],\quad \forall (t,\mathbf{x})\in[0,T]\times\overline{\R}_+^d,
\end{align}
where $\mathbf{X}^{t,x}=(X_s^{1,t,x_1},\ldots,X_s^{d,t,x_d})_{s\in[t,T]}$ is the state process whose $i$-th component is given by the reflected Brownian motion with drift that, for $i=1,\ldots,d$,
\begin{align}\label{eq:RSDE0}
X_s^{i,t,x_i} &=x_i +  \int_t^s \mu_i dr + \sum_{k=1}^m \int_t^s \sigma_{ik} dB_r^k + L_s^{i,t}\in\overline{\R}_+,\quad \forall s\in[t,T].
\end{align}
Here, on a filtered probability space $(\Omega,\F,\Fx,\mathbb{P})$ with the filtration $\Fx=(\F_t)_{t\in[0,T]}$ satisfying usual conditions, $B=(B_s^1,\ldots,B_s^m)_{s\in[t,T]}$ is an $m$-dimensional Brownian motion and $s\to L_s^{i,t}$ is a continuous and non-decreasing process (with $L_t^{i,t}=0$) which increases on the time set $\{s\in[t,T];~X_s^{i,t,x_i}=0\}$ only. In the context of queueing system, the state process $\mathbf{X}^{t,x}$ with a possibly general reflected-direction matrix arises in heavy traffic theory for $d$-station networks of queues (c.f. \citealt{Harrison81}).  We denote the vector process $\mathbf{X}^{-i,t,\mathbf{x}}:=(X_s^{1,t,x_1},\ldots,X_{s}^{i-1,t,x_{i-1}},X_{s}^{i+1,t,x_{i+1}},\ldots,X_{s}^{d,t,x_{d}})_{s\in[t,T]}$.

The aim of this paper is to establish the existence and uniqueness of a classical solution to the Robin boundary problem \eqref{eq:PDE-multidim} of parabolic type. A naive way is to directly verify that the probabilistic representation in \eqref{eq:PRprimalNeumann} is a classical solution to the Robin boundary problem \eqref{eq:PDE-multidim} by applying the stochastic flow technique. However, we stress that, the integral in \eqref{eq:PRprimalNeumann} inside the expectation is with respect to a local time process $L^{i,t}$, which is not independent of the reflected state process $\mathbf{X}^{-i,t,\mathbf{x}}$. It is, in general, difficult to prove the regularity of the probabilistic representation $(t,\mathbf{x})\to u(t,\mathbf{x})$. To overcome this challenge, we contribute in this paper by proposing a new decomposition-homogenization for the Robin boundary problem \eqref{eq:PDE-multidim}, which yields two auxiliary Robin boundary problems whose probabilistic solutions can be verified to be classical solutions using some stochastic flow techniques. 

Let us first state the main result of this paper, whose proof is given in Section \ref{sec:Neumann-inde}.

\begin{theorem}\label{thm:PDE-sol0}
The Robin boundary problem \eqref{eq:PDE-multidim} has a unique classical solution $u\in C^{1,2}([0,T)\times\R_+^d)\cap C([0,T]\times\overline{\R}_+^d)$ satisfying the polynomial growth condition (i.e., there exists a constant $C>0$ and $q\geq 1$ such that $|u(t,\mathbf{x})|\leq C(1+|\mathbf{x}|^q)$). Moreover, this classical solution $u(t,\mathbf{x})$ admits the  probabilistic representation that, for all $(t,\mathbf{x})\in[0,T]\times\overline{\R}_+^d$,
\begin{align}\label{eq:thm-PRprimalNeumann}
u(t,\mathbf{x}) &= \varphi(t,\mathbf{x})+\frac{1}{2}\sum_{i\neq j}\Ex\left[\int_t^{T} e^{-\rho(s-t)+\sum_{k=1}^d c_k L_s^{k,t}}\left(\sum_{k=1}^m\sigma_{ik}\sigma_{jk}\right)\partial_{x_ix_j}^2\varphi(s,\mathbf{X}_s^{t,\mathbf{x}})ds\right].\nonumber\\
&\quad+\Ex\left[e^{-\rho (T-t)+\sum_{k=1}^d c_k L_T^{k,t}} g(\mathbf{X}_T^{t,\mathbf{x}})\right],
\end{align}
where $\varphi\in C^{1,2}([0,T)\times\R_+^d)\cap C([0,T]\times\overline{\R}_+^d)$ admits the probabilistic representation:
\begin{align}\label{eq:varphiinde}
\varphi(t,\mathbf{x}) &=-\sum_{i=1}^d\left\{\frac{{\bf 1}_{c_i\neq 0}}{c_i}\int_t^{T} e^{-\rho(s-t)}\Ex\left[e^{\sum_{j\neq i}c_j \hat{L}_s^{j,t}}f_i(s,\hat{\mathbf{X}}_s^{-i,t,\mathbf{x}})\right]d\Ex[e^{c_i \hat{L}_s^{i,t}}]\right.\\
&\qquad\qquad\left.+{\bf 1}_{c_i=0}\int_t^{T} e^{-\rho(s-t)}\Ex\left[e^{\sum_{j\neq i}c_j \hat{L}_s^{j,t}}f_i(s,\hat{\mathbf{X}}_s^{-i,t,\mathbf{x}})\right]d\Ex[\hat{L}_s^{i,t}]\right\}.\nonumber
\end{align}
Here, the $i$-th component of the state process $\hat{\mathbf{X}}^{t,\mathbf{x}}=(\hat{X}_s^{1,t,x_1},\ldots,\hat{X}_{s}^{d,t,x_{d}})_{s\in[t,T]}$ is the reflected Brownian motion with drift that, for all $(t,\mathbf{x})\in[0,T]\times\overline{\R}_+^d$,
\begin{align}\label{eq:RSDEi}
\hat{X}_s^{i,t,x_i}&=x_i +  \int_t^s \mu_i dr + \sum_{k=1}^m \int_t^T \sigma_{ik}dW_s^{k,i} + \hat{L}_s^{i,t}\in\overline{\R}_+,
\end{align}
where, on the filtered probability space $(\Omega,\F,\Fx,\mathbb{P})$, $W^{k,i}=(W_t^{k,i})_{t\in[0,T]}$ for $(k,i)\in\{1,\ldots,m\}\times\{1,\ldots,d\}$ are independent Brownian motions, and $s\to\hat{L}_s^{i,t}$ (with $\hat{L}_t^{i,t}=0$) is a continuous and non-decreasing process that increases on the time set $\{s\in[t,T];~\hat{X}_s^{i,t,x_i}=0\}$ only.  In \eqref{eq:varphiinde}, the vector process is given by
\begin{align*}
\hat{\mathbf{X}}^{-i,t,\mathbf{x}}:=(\hat{X}_s^{1,t,x_1},\ldots,\hat{X}_{s}^{i-1,t,x_{i-1}},\hat{X}_{s}^{i+1,t,x_{i+1}},\ldots,\hat{X}_{s}^{d,t,x_{d}})_{s\in[t,T]},~~i=1,\ldots,d.    
\end{align*}
\end{theorem}

The remainder of this paper is organized as follows. Section~\ref{sec:method} describes our decomposition-homogenization method for the Robin boundary problem \eqref{eq:PDE-multidim} in the nonnegative orthant. Section~\ref{sec:Neumann-inde} establishes the smoothness of probabilistic solutions to two auxiliary Robin boundary problems using some stochastic flow techniques and concludes Theorem~\ref{thm:PDE-sol0}. The proofs of auxiliary results in previous sections are reported in Section~\ref{appendix:proof}.

\section{Decomposition and Homogenization Method}\label{sec:method}

In this section, we introduce the decomposition-homogenization method to solve the Robin boundary problem \eqref{eq:PDE-multidim}. This method yields two auxiliary Robin boundary problems of parabolic PDEs, and the probabilistic representations of the two auxiliary Robin boundary problems have ``nice" structures, which can be verified to be classical solutions by applying stochastic flow arguments.

Let us first consider the following auxiliary Robin boundary problem that
\begin{align}\label{eq:PDE-multidimL0}
\begin{cases}
\displaystyle (\partial_t+\mathcal{L}^0)\varphi(t,\mathbf{x}) = \rho\varphi(t,\mathbf{x}),\quad \text{on}~(t,\mathbf{x})\in[0,T)\times\R_+^d,\\[0.6em]
\displaystyle \varphi(T,\mathbf{x})=0,\quad \forall \mathbf{x}\in\overline{\R}_+^d,\\[0.6em]
\displaystyle \partial_{x_i}\varphi(t,(\mathbf{x}^{-i},0))+c_i \varphi(t,(\mathbf{x}^{-i},0))=f_i(t,\mathbf{x}^{-i}),\quad \forall (t,\mathbf{x}^{-i})\in[0,T)\times\overline{\R}_+^{d-1},\\[0.4em]
\displaystyle \qquad\qquad\qquad\qquad\qquad\qquad\qquad\qquad\qquad\qquad\quad~i=1,\ldots,d.
\end{cases}
\end{align}
Here, the second-order differential operator $\mathcal{L}^0$ acting on $\phi\in C^2(\R_+^d)$ is defined by
\begin{align}\label{eq:Lmulti0}
\mathcal{L}^0\phi(\mathbf{x}) &:= \frac{1}{2}\sum_{i=1}^d \left(\sum_{k=1}^m\sigma_{ik}^2\right)\partial_{x_i}^2\phi(\mathbf{x}) +\sum_{i=1}^d\mu_i\partial_{x_i}\phi(\mathbf{x}).
\end{align}
We assume that the auxiliary Robin boundary problem \eqref{eq:PDE-multidimL0} above has a classical solution $\varphi:[0,T]\times\overline{\R}_+^d\to\R$. Then, we introduce the next Robin boundary problem with homogeneous Robin boundary conditions that
\begin{align}\label{eq:PDE-multidim-homo}
\begin{cases}
\displaystyle (\partial_t+\mathcal{L})\psi(t,\mathbf{x}) = \rho\psi(t,\mathbf{x})-\frac{1}{2}\sum_{i\neq j}\left(\sum_{k=1}^m\sigma_{ik}\sigma_{jk}\right)\partial_{x_ix_j}^2\varphi(t,\mathbf{x}),\\[0.8em]
\displaystyle\qquad\qquad\qquad\qquad\qquad\qquad\qquad \text{on}~(t,\mathbf{x})\in[0,T)\times\R_+^d,\\[0.8em]
\displaystyle \psi(T,\mathbf{x}) = g(\mathbf{x}),\quad \forall \mathbf{x}\in\overline{\R}_+^d,\\[0.6em]
\displaystyle \partial_{x_i}\psi(t,(\mathbf{x}^{-i},0))+c_i \psi(t,(\mathbf{x}^{-i},0))=0,\quad \forall (t,\mathbf{x}^{-i})\in[0,T)\times\overline{\R}_+^{d-1},\\[0.4em]
\displaystyle \qquad\qquad\qquad\qquad\qquad\qquad\qquad\qquad\qquad~i=1,\ldots,d.
\end{cases}
\end{align}
We observe that the auxiliary Robin boundary problem \eqref{eq:PDE-multidimL0} has the principal operator $\mathcal{L}^0$ satisfying
\begin{align}\label{eq:differenceLL0}
(\mathcal{L}-\mathcal{L}^0)\phi=\frac{1}{2}\sum_{i\neq j}\left(\sum_{k=1}^m\sigma_{ik}\sigma_{jk}\right)\partial_{x_ix_j}^2\phi,\quad \forall \phi\in C^2(\R_+^d),
\end{align}
and the auxiliary Robin boundary problem \eqref{eq:PDE-multidim-homo} has homogeneous Robin boundary conditions. The next result establishes the relationship between the solution to the original Robin boundary problem \eqref{eq:PDE-multidim} and the solutions to two auxiliary Robin boundary problems \eqref{eq:PDE-multidimL0} and \eqref{eq:PDE-multidim-homo}.

\begin{lemma}[Decomposition-Homogenization Method]\label{lem:decomposition}
Let $\varphi$ be a classical solution to the auxiliary Robin boundary problem \eqref{eq:PDE-multidimL0}. Given this classical solution $\varphi$, let $\psi$ be a classical solution to the auxiliary Robin boundary problem \eqref{eq:PDE-multidim-homo}. Then $u=\varphi+\psi$ is a classical solution to the original Robin boundary problem \eqref{eq:PDE-multidim}.
\end{lemma}

\begin{proof}
It follows from \eqref{eq:PDE-multidimL0}, \eqref{eq:PDE-multidim-homo} and \eqref{eq:differenceLL0} that, on $[0,T)\times\R_+^{d}$,
\begin{align*}
(\partial_t+\mathcal{L}^0)u&=(\partial_t+\mathcal{L}^0)\varphi+(\partial_t+\mathcal{L}^0)\psi=\rho\varphi+(\partial_t+\mathcal{L})\psi-\frac{1}{2}\sum_{i\neq j}\left(\sum_{k=1}^m\sigma_{ik}\sigma_{jk}\right)\partial_{x_ix_j}^2\psi\nonumber\\
&=\rho\varphi+\rho\psi-\frac{1}{2}\sum_{i\neq j}\left(\sum_{k=1}^m\sigma_{ik}\sigma_{jk}\right)\partial_{x_ix_j}^2\varphi-\frac{1}{2}\sum_{i\neq j}\left(\sum_{k=1}^m\sigma_{ik}\sigma_{jk}\right)\partial_{x_ix_j}^2\psi\nonumber\\
&=\rho u - \frac{1}{2}\sum_{i\neq j}\left(\sum_{k=1}^m\sigma_{ik}\sigma_{jk}\right)\partial_{x_ix_j}^2u.
\end{align*}
Then, it follows from \eqref{eq:differenceLL0} again that, on $[0,T)\times\R_+^{d}$,
\begin{align*}
(\partial_t+\mathcal{L})u=(\partial_t+\mathcal{L}^0)u+\frac{1}{2}\sum_{i\neq j}\left(\sum_{k=1}^m\sigma_{ik}\sigma_{jk}\right)\partial_{x_ix_j}^2u=\rho u.
\end{align*}
On the other hand, for all $(t,\mathbf{x}^{-i})\in[0,T)\times\overline{\R}_+^{d-1}$ with $i=1,\ldots,d$,
\begin{align*}
&\partial_{x_i}u(t,(\mathbf{x}^{-i},0))+c_i u(t,(\mathbf{x}^{-i},0))\nonumber\\
&\qquad=\partial_{x_i}\varphi(t,(\mathbf{x}^{-i},0))+c_i \varphi(t,(\mathbf{x}^{-i},0))
+\partial_{x_i}\psi(t,(\mathbf{x}^{-i},0))+c_i \psi(t,(\mathbf{x}^{-i},0))\nonumber\\
&\qquad=f_i(t,\mathbf{x}^{-i})+0=f_i(t,\mathbf{x}^{-i}).
\end{align*}
It follows that $u=\varphi+\psi$ satisfies the boundary conditions in the original Robin boundary problem \eqref{eq:PDE-multidim}. Finally, we have that, for all $\mathbf{x}\in\overline{\R}_+^{d}$,
\begin{align*}
u(T,\mathbf{x})=\varphi(T,\mathbf{x})+\psi(T,\mathbf{x})=0+g(\mathbf{x})=g(\mathbf{x}).
\end{align*}
This verifies that $u$ satisfies the terminal condition in the original Robin boundary problem \eqref{eq:PDE-multidim}. Thus, we conclude that $u=\varphi+\psi$ is a classical solution to the Robin boundary problem \eqref{eq:PDE-multidim}.
\end{proof}

In lieu of Lemma~\ref{lem:decomposition}, in order to prove the smoothness of the probabilistic representation \eqref{eq:PRprimalNeumann}, we can verify the smoothness of probabilistic representations of solutions to the auxiliary Robin boundary problems \eqref{eq:PDE-multidimL0} and \eqref{eq:PDE-multidim-homo}, respectively. We find that it is more feasible to show their smoothness by using some stochastic flow techniques, which will be elaborated in detail in the next section.

\section{Probabilistic Solutions to Auxiliary Robin Boundary Problems}\label{sec:Neumann-inde}

\subsection{Robin boundary problem \eqref{eq:PDE-multidimL0}}\label{sec:Neumann1}

This subsection focuses on the construction of the smooth probabilistic solution to the auxiliary Robin boundary problem \eqref{eq:PDE-multidimL0}. For reader's convenience, we recall the Robin boundary problem \eqref{eq:PDE-multidimL0} that
\begin{align*}
\begin{cases}
\displaystyle (\partial_t+\mathcal{L}^0)\varphi(t,\mathbf{x}) = \rho\varphi(t,\mathbf{x}),\quad \text{on}~(t,\mathbf{x})\in[0,T)\times\R_+^d,\\[0.6em]
\displaystyle \varphi(T,\mathbf{x})=0,\quad \forall \mathbf{x}\in\overline{\R}_+^d,\\[0.6em]
\displaystyle \partial_{x_i}\varphi(t,(\mathbf{x}^{-i},0))+c_i \varphi(t,(\mathbf{x}^{-i},0))=f_i(t,\mathbf{x}^{-i}),\quad \forall (t,\mathbf{x}^{-i})\in[0,T)\times\overline{\R}_+^{d-1},\\[0.4em]
\displaystyle \qquad\qquad\qquad\qquad\qquad\qquad\qquad\qquad\qquad\qquad\quad~i=1,\ldots,d.
\end{cases}
\end{align*}
We next introduce the probabilistic representation of the solution to Robin boundary problem \eqref{eq:PDE-multidimL0}, for all $(t,\mathbf{x})\in[0,T]\times\overline{\R}_+^d$,
\begin{align}\label{eq:varphisol}
\varphi(t,\mathbf{x}) &=-\sum_{i=1}^d\Ex\left[\int_t^{T} e^{-\rho(s-t)+\sum_{k=1}^d c_k\hat{L}_s^{k,t}}f_i(s,\hat{\mathbf{X}}_s^{-i,t,\mathbf{x}})d\hat{L}_s^{i,t}\right],
\end{align}
where, we recall from \eqref{eq:RSDEi} that
\begin{align*}
\hat{X}_s^{i,t,x_i}&=x_i +  \int_t^s \mu_i dr + \sum_{k=1}^m \int_t^T \sigma_{ik}dW_s^{k,i} + \hat{L}_s^{i,t}\in\overline{\R}_+.
\end{align*}
Here, $s\to\hat{L}_s^{i,t}$ (with $\hat{L}_t^{i,t}=0$) is a continuous and non-decreasing process which increases on the time set $\{s\in[t,T];~\hat{X}_s^{i,t,x_i}=0\}$ only.

A key observation to the probabilistic representation \eqref{eq:varphisol} is that the state process $\hat{\mathbf{X}}^{-i,t,\mathbf{x}}$ is in fact independent of the local time process $\hat{L}^{i,t}=(\hat{L}_s^{i,t})_{s\in[t,T]}$. Then, we have the following result whose proof is reported in Section~\ref{appendix:proof}.
\begin{lemma}\label{lem:varphiprobab0}
The probabilistic representation $\varphi(t,\mathbf{x})$ given by \eqref{eq:varphisol} can be rewritten as follows, for all $(t,\mathbf{x})\in[0,T]\times\overline{\R}_+^d$,
\begin{align*}
\varphi(t,\mathbf{x}) &=-\sum_{i=1}^d\left\{\frac{{\bf 1}_{c_i\neq 0}}{c_i}\int_t^{T} e^{-\rho(s-t)}\Ex\left[e^{\sum_{j\neq i}c_j \hat{L}_s^{j,t}}f_i(s,\hat{\mathbf{X}}_s^{-i,t,\mathbf{x}})\right]d\Ex[e^{c_i \hat{L}_s^{i,t}}]\right.\\
&\qquad\qquad\left.+{\bf 1}_{c_i=0}\int_t^{T} e^{-\rho(s-t)}\Ex\left[e^{\sum_{j\neq i}c_j \hat{L}_s^{j,t}}f_i(s,\hat{\mathbf{X}}_s^{-i,t,\mathbf{x}})\right]d\Ex[\hat{L}_s^{i,t}]\right\}.\nonumber
\end{align*}
\end{lemma}
Building upon the above probabilistic representation of $\varphi(t,\mathbf{x})$ provided in Lemma~\ref{lem:varphiprobab0}, the next lemma establishes the smoothness of $(t,\mathbf{x})\to\varphi(t,\mathbf{x})$ given by \eqref{eq:varphisol} using the stochastic flow technique.

\begin{lemma}\label{lem:regularity-varphi}
Consider the probabilistic representation of $\varphi(t,\bf{x})$ given in Lemma~\ref{lem:varphiprobab0}. Then, the probabilistic representation $\varphi\in C^{1,2}([0,T)\times\R_+^d)\cap C([0,T]\times\overline{\R}_+^d)$. Moreover, it holds that, for all $i=1,\ldots,d$,
\begin{align}\label{eq:varphi-xj}
 \partial_{x_i} \varphi(t,\bf{x})&=\Ex\left[e^{-\rho (\tau_{x_i}^i-t)+\sum_{k\neq i}c_k \hat{L}_{\tau_{x_i}^i}^{k,t}}f_i\left(\tau_{x_i}^i,\hat{\bf{X}}_{\tau_{x_i}^i}^{-i,t,\bf{x}}\right){\bf 1}_{\tau_{x_i}^i<T}\right]\nonumber\\
&\quad-\sum_{j\neq i} \mathbb{E}\left[\int_t^{\tau^i_{x_i}\wedge T}e^{-\rho (s-t)+\sum_{k\neq i}c_k \hat{L}_s^{k,t}} \partial_{x_i} f_j\left(s,\hat{\bf{X}}_s^{-j,t,\bf{x}}\right) d \hat{L}_s^{j,t}\right]\nonumber\\
&\quad+c_i\sum_{j=1}^d \mathbb{E}\left[\int_{\tau_{x_i}^i\wedge T}^{T}  e^{-\rho (s-t)+\sum_{k=1}^d c_k \hat{L}_s^{k,t}}f_j\left(s,\hat{\bf{X}}_s^{-j,t,\bf{x}}\right) d\hat{L}_s^{j,t}\right],
\end{align}
where, for $i=1,\ldots,d$, and $x\in\overline{\R}_+$, the stopping time $\tau_{x}^i$ is defined by
\begin{align}\label{tau-i}
\tau_{x}^i:=\inf\left\{s\geq t;~-\mu_i (s-t) -\sum_{k=1}^m \sigma_{ik}(W_s^{k,i}-W_t^{k,i})=x\right\}.
\end{align}
\end{lemma}

The proof of Lemma~\ref{lem:regularity-varphi} is reported in Section~\ref{appendix:proof}.  The well-posedness of the decomposed Robin boundary problem \eqref{eq:PDE-multidimL0} is given in the next result:
\begin{proposition}\label{prop:varphisol}
Consider the probabilistic representation $\varphi(t,\bf{x})$ given by \eqref{eq:varphisol}. Then, we have
\begin{itemize}
\item[{\rm(i)}] the probabilistic representation $\varphi$ is a classical solution of the auxiliary Robin boundary problem \eqref{eq:PDE-multidimL0}.
\item[{\rm(ii)}] if the auxiliary Robin boundary problem \eqref{eq:PDE-multidimL0} has a classical solution $\varphi$ satisfying the polynomial growth condition (i.e., there exists a constant $C>0$ and $q\geq 1$ such that $|\varphi(t,\mathbf{x})|\leq C(1+|\mathbf{x}|^q)$), then the solution $\varphi$ admits the probabilistic representation given by \eqref{eq:varphisol}.
\end{itemize}
\end{proposition}
The proof of Proposition \ref{prop:varphisol} is reported in Section~\ref{appendix:proof}.

\subsection{Robin boundary problem~\eqref{eq:PDE-multidim-homo}}

Section~\ref{sec:Neumann1} has shown that the probabilistic representation $\varphi$ given by \eqref{eq:varphisol} is a classical solution to the Robin boundary problem \eqref{eq:PDE-multidimL0}. Given this classical solution $\varphi$, this section examines the construction of a smooth probabilistic solution to Robin boundary problem \eqref{eq:PDE-multidim-homo}, i.e.,
\begin{align*}
\begin{cases}
\displaystyle (\partial_t+\mathcal{L})\psi(t,\mathbf{x}) = \rho\psi(t,\mathbf{x})-\frac{1}{2}\sum_{i\neq j}\left(\sum_{k=1}^m\sigma_{ik}\sigma_{jk}\right)\partial_{x_ix_j}^2\varphi(t,\mathbf{x}),\\[0.8em]
\displaystyle\qquad\qquad\qquad\qquad\qquad\qquad\qquad \text{on}~(t,\mathbf{x})\in[0,T)\times\R_+^d,\\[0.8em]
\displaystyle \psi(T,\mathbf{x}) = g(\mathbf{x}),\quad \forall \mathbf{x}\in\overline{\R}_+^d,\\[0.6em]
\displaystyle \partial_{x_i}\psi(t,(\mathbf{x}^{-i},0))+c_i \psi(t,(\mathbf{x}^{-i},0))=0,\quad \forall (t,\mathbf{x}^{-i})\in[0,T)\times\overline{\R}_+^{d-1},\\[0.4em]
\displaystyle \qquad\qquad\qquad\qquad\qquad\qquad\qquad\qquad\qquad~i=1,\ldots,d.
\end{cases}
\end{align*}
To this end, let us introduce the following probabilistic representation that, for all $(t,\mathbf{x})\in[0,T]\times\overline{\R}_+^d$,
\begin{align}\label{eq:varphisol2}
\psi(t,\mathbf{x}) &=\frac{1}{2}\sum_{i\neq j}\Ex\left[\int_t^{T} e^{-\rho(s-t)+\sum_{k=1}^d c_k L_s^{k,t}}\left(\sum_{k=1}^m\sigma_{ik}\sigma_{jk}\right)\partial_{x_ix_j}^2\varphi(s,\mathbf{X}_s^{t,\mathbf{x}})ds\right]\nonumber\\
&\quad+\Ex\left[e^{-\rho(T-t)+\sum_{k=1}^d c_k L_T^{k,t}} g(\mathbf{X}_T^{t,\mathbf{x}})\right].
\end{align}
Here, we recall that the  $i$-th component of state process $\mathbf{X}^{t,\mathbf{x}}=(X_s^{1,t,x_1},\ldots,X_{s}^{d,t,x_d})_{s\in[t,T]}$ obeys the reflected Brownian motion with drift in \eqref{eq:RSDE0}.

From the expression of $\psi(t,{\bf x})$ in \eqref{eq:varphisol2}, we note that the integral in the first expectation in  \eqref{eq:varphisol2} is the one with respect to the Lebesgue measure ``$ds$". Thus, it becomes much easier to apply the stochastic flow argument to show the smoothness of the probabilistic representation $(t,{\bf x})\to\psi(t,{\bf x})$. On the other hand, in order to check the smoothness of $(t,{\bf x})\to\psi(t,{\bf x})$, we also need a flexible form for the second order partial derivative $\partial_{x_ix_j}^2\varphi(t,{\bf x})$. This relies on the probabilistic form \eqref{eq:varphiinde} for $\varphi(t,{\bf x})$ and Proposition \ref{prop:varphisol}. In fact, we can apply some stochastic flow arguments to obtain the  representation of $\partial_{x_ix_j}^2\varphi(t,{\bf x})$ in the next lemma, whose proof is delegated to Section~\ref{appendix:proof}.

\begin{lemma}\label{lem:varphixixj}
 Recall that the probabilistic representation $\varphi(t,\mathbf{x})$ given by \eqref{eq:varphisol} is a classical solution to the Robin boundary problem \eqref{eq:PDE-multidimL0}, which has been shown in Proposition \ref{prop:varphisol}. For any $\ell=1,\ldots,d$, let us introduce the function $\varphi_{\ell}:[0,T]\times\overline{\R}_+^2\to\R$ that, for all $(t,{\bf x})\in[0,T]\times\overline{\R}^d_+$,
\begin{align}\label{varphi-m}
\displaystyle\varphi_{\ell}(t,{\bf x}):=-\Ex\left[\int_t^{T} e^{-\rho(s-t)+\sum_{k=1}^d c_k \hat{L}_s^{k,t}}f_{\ell}(s,\hat{\mathbf{X}}_s^{-\ell,t,\mathbf{x}})d\hat{L}_s^{\ell,t}\right].
\end{align}
Then, for all $1\leq i<j\leq d$, we have
\begin{align*}
\partial_{x_ix_j}^2\varphi(t,{\bf x})=\sum_{\ell=1}^d\partial_{x_ix_j}^2 \varphi_{\ell}(t,{\bf x}).
\end{align*}
Moreover, it holds that, for $\ell=i$ or $\ell=j$,
\begin{align*}
&\partial_{x_\ell x_j}^2\varphi_\ell (t,\bf{x})\\
&=c_j\int_t^{T}\int_{0}^{\infty}\int_{-\infty}^{y_1}\cdots\int_{0}^{\infty}\int_{-\infty}^{y_{\ell-1}}\int_{0}^{\infty}\int_{-\infty}^{y_{\ell+1}}\cdots\int_{x_j}^{\infty}\int_{-\infty}^{y_j}\cdots\int_{0}^{\infty}\int_{-\infty}^{y_d}e^{-\rho(s-t)+\sum_{k\neq \ell}^d c_k (y_k-x_k)}\\
&\quad\times f_\ell(s,{\bf x}^{-\ell}-{\bf r}^{-\ell}+({\bf y}^{-\ell}-{\bf x}^{-\ell})^+)\partial_{x_\ell} h_\ell(t,s,x_\ell) \Pi_{k\neq \ell}^d\phi_k(t,s,r_k,y_k) \Pi_{k\neq \ell}(dr_kdy_k)ds\\
&\quad-\int_t^{T}\int_{0}^{\infty}\int_{-\infty}^{y_1}\cdots\int_{0}^{\infty}\int_{-\infty}^{y_{\ell-1}}\int_{0}^{\infty}\int_{-\infty}^{y_{\ell+1}}\cdots\int_0^{x_j}\int_{-\infty}^{y_j}\cdots\int_{0}^{\infty}\int_{-\infty}^{y_d}e^{-\rho(s-t)+\sum_{k\neq \ell}^d c_k (y_k-x_k)}\\
&\qquad\times \partial_{x_j}f_\ell(s,{\bf x}^{-\ell}-{\bf r}^{-\ell}+({\bf y}^{-\ell}-{\bf x}^{-\ell})^+)\partial_{x_\ell} h_\ell(t,s,x_\ell) \Pi_{k\neq \ell}^d\phi_k(t,s,r_k,y_k) \Pi_{k\neq \ell}(dr_kdy_k)ds.
\end{align*}For $\ell\neq i$ and $\ell\neq j$, we have that
\begin{align*}
&\partial_{x_ix_j}^2\varphi_\ell(t,\bf{x})\\
&=-c_ic_j\int_t^{T}\int_{0}^{\infty}\int_{-\infty}^{y_1}\cdots\int_{0}^{\infty}\int_{-\infty}^{y_{\ell-1}}\int_{0}^{\infty}\int_{-\infty}^{y_{\ell+1}}\cdots\int_{x_i}^{\infty}\int_{-\infty}^{y_i}\cdots\int_{x_j}^{\infty}\int_{-\infty}^{y_j}\cdots\int_{0}^{\infty}\int_{-\infty}^{y_d}e^{-\rho(s-t)}\\
&\times e^{\sum_{k\neq \ell}^d c_k (y_k-x_k)} f_\ell(s,{\bf x}^{-\ell}-{\bf r}^{-\ell}+({\bf y}^{-\ell}-{\bf x}^{-\ell})^+) h_\ell(t,s,x_\ell)  \Pi_{k\neq \ell}^d\phi_k(t,s,r_k,y_k) \Pi_{k\neq \ell}(dr_kdy_k)ds\\
&+c_i\int_t^{T}\int_{0}^{\infty}\int_{-\infty}^{y_1}\cdots\int_{0}^{\infty}\int_{-\infty}^{y_{\ell-1}}\int_{0}^{\infty}\int_{-\infty}^{y_{\ell+1}}\cdots\int_0^{x_i}\int_{-\infty}^{y_i}\cdots\int_{x_j}^{\infty}\int_{-\infty}^{y_j}\cdots\int_{0}^{\infty}\int_{-\infty}^{y_d}e^{-\rho(s-t)}\\
&\times e^{\sum_{k=2}^d c_k (y_k-x_k)}\partial_{x_i}f_\ell(s,{\bf x}^{-\ell}-{\bf r}^{-\ell}+({\bf y}^{-\ell}-{\bf x}^{-\ell})^+)h_\ell(t,s,x_\ell)   \Pi_{k\neq \ell}^d\phi_k(t,s,r_k,y_k) \Pi_{k\neq \ell}(dr_kdy_k)ds\\
&+c_j\int_t^{T}\int_{0}^{\infty}\int_{-\infty}^{y_1}\cdots\int_{0}^{\infty}\int_{-\infty}^{y_{\ell-1}}\int_{0}^{\infty}\int_{-\infty}^{y_{\ell+1}}\cdots\int_{x_i}^{\infty}\int_{-\infty}^{y_i}\cdots\int_0^{x_j}\int_{-\infty}^{y_j}\cdots\int_{0}^{\infty}\int_{-\infty}^{y_d}e^{-\rho(s-t)}\\
&\times e^{\sum_{k=2}^d c_k (y_k-x_k)}\partial_{x_j}f_\ell(s,{\bf x}^{-\ell}-{\bf r}^{-\ell}+({\bf y}^{-\ell}-{\bf x}^{-\ell})^+)h_\ell(t,s,x_\ell)  \Pi_{k\neq \ell}^d\phi_k(t,s,r_k,y_k) \Pi_{k\neq \ell}(dr_kdy_k)ds\\
&-\int_t^{T}\int_{0}^{\infty}\int_{-\infty}^{y_1}\cdots\int_{0}^{\infty}\int_{-\infty}^{y_{\ell-1}}\int_{0}^{\infty}\int_{-\infty}^{y_{\ell+1}}\cdots\int_0^{x_i}\int_{-\infty}^{y_i}\cdots\int_0^{x_j}\int_{-\infty}^{y_j}\cdots\int_{0}^{\infty}\int_{-\infty}^{y_d}e^{-\rho(s-t)}\\
&\times e^{\sum_{k=2}^d c_k (y_k-x_k)}\partial_{x_ix_j}f_\ell(s,{\bf x}^{-\ell}-{\bf r}^{-\ell}+({\bf y}^{-\ell}-{\bf x}^{-m})^+) h_\ell(t,s,x_\ell)  \Pi_{k\neq \ell}^d\phi_k(t,s,r_k,y_k) \Pi_{k\neq \ell}(dr_kdy_k)ds.
\end{align*}
For $\ell=1,\ldots,d$, the functions $h_\ell(t,s,x)$ and $\phi_\ell(t,s,x,y)$ are respectively given by \eqref{eq:function-h} and \eqref{eq:phii}.
\end{lemma}

By applying Lemma \ref{lem:varphixixj}, we first have
\begin{lemma}\label{lem:regularity-psi}
Consider the probabilistic representation of $\psi(t,\bf{x})$ given by \eqref{eq:varphisol2}. Then $\psi\in C^{1,2}([0,T)\times\R_+^d)\cap C([0,T]\times\overline{\R}_+^d)$.
\end{lemma}
The proof of Lemma~\ref{lem:regularity-psi} is similar to the one of Lemma \ref{lem:regularity-varphi} by combining Lemma~\ref{lem:varphixixj}
and stochastic flow techniques. We hence omit the proof of Lemma~\ref{lem:regularity-psi}. Building upon Lemma \ref{lem:regularity-psi}, we now provide the well-posedness of the Robin boundary problem \eqref{eq:PDE-multidim-homo} in the next result, whose proof is similar to that of Proposition \ref{prop:varphisol}.
\begin{proposition}\label{prop:varphisol2}
Consider the probabilistic representation $\psi(t,\bf{x})$ given by \eqref{eq:varphisol2}. Then, we have
\begin{itemize}
\item[{\rm(i)}] the probabilistic representation $\psi$ is a classical solution of the auxiliary Robin boundary problem \eqref{eq:PDE-multidim-homo}.
\item[{\rm(ii)}] if the auxiliary Robin boundary problem \eqref{eq:PDE-multidim-homo} has a classical solution $\psi$ satisfying the polynomial growth condition (i.e.,  there exists a constant $C>0$ and $q\geq 1$ such that $|\psi(t,\mathbf{x})|\leq C(1+|\mathbf{x}|^q)$), then the solution $\psi$ admits the probabilistic representation given by \eqref{eq:varphisol2}.
\end{itemize}
\end{proposition}

We are now ready to prove Theorem \ref{thm:PDE-sol0}.
\begin{proof}[Proof of Theorem \ref{thm:PDE-sol0}]
By applying Lemma \ref{lem:decomposition}, Proposition \ref{prop:varphisol} and Proposition \ref{prop:varphisol2}, we have that, for all $(t,\mathbf{x})\in[0,T]\times\overline{\R}_+^d$, $\tilde{u}(t,\mathbf{x})=\varphi(t,\mathbf{x})+\psi(t,\mathbf{x})$ solves the Robin boundary problem \eqref{eq:PDE-multidim}, which provides the existence of a classical solution to the Robin boundary problem \eqref{eq:PDE-multidim}. On the other hand, assuming that $u\in C^{1,2}([0,T)\times \R_+^d)\cap C([0,T]\times \overline{\R}_+^d)$ with the polynomial growth condition is a classical solution of the Robin boundary problem \eqref{eq:PDE-multidim}, by a similar argument in the proof of Proposition \ref{prop:varphisol}-(ii), we can deduce that $u$ admits the probabilistic representation \eqref{eq:PRprimalNeumann}. This also yields the uniqueness of the classical solution of the Robin boundary problem \eqref{eq:PDE-multidim}. Finally, noting that $\tilde{u}(t,\mathbf{x})=\varphi(t,\mathbf{x})+\psi(t,\mathbf{x})$ is also a classical solution to the Robin boundary problem \eqref{eq:PDE-multidim}, by the uniqueness of the classical solution, we deduce that $u$ satisfies $u(t,\mathbf{x})=\tilde{u}(t,\mathbf{x})=\varphi(t,\mathbf{x})+\psi(t,\mathbf{x})$ for all $(t,\mathbf{x})\in[0,T]\times\overline{\R}_+^d$, which implies the equivalence between the probabilistic representations \eqref{eq:PRprimalNeumann} and \eqref{eq:thm-PRprimalNeumann}. Thus, we complete the proof.
\end{proof}

\begin{remark}
 In view of Lemma~\ref{lem:decomposition}, Proposition \ref{prop:varphisol} and  \ref{prop:varphisol2}, our decomposition-homogenization technique for the Robin boundary problem \eqref{eq:PDE-multidim} can be also applied to the following Robin boundary problem of elliptic type:
\begin{align}\label{eq:ellipticPDE-multidim}
\begin{cases}
\displaystyle \mathcal{L}u(\mathbf{x}) = \rho u(\mathbf{x}),\quad \text{on}~\mathbf{x}\in\R_+^d,\\[0.6em]
\displaystyle \partial_{x_i}u(\mathbf{x}^{-i},0)+c_i u(\mathbf{x}^{-i},0)=f_i(\mathbf{x}^{-i}),\quad \forall \mathbf{x}^{-i}\in\overline{\R}_+^{d-1},~i=1,\ldots,d,
\end{cases}
\end{align}
where $\rho>0$, $c_1,...,c_d$ are non-positive constants and $f,\partial_{x_j}f$, $\partial_{x_jx_k}^2f$ are bounded for  $j,k=1,\ldots,d$. The Robin boundary problem \eqref{eq:ellipticPDE-multidim} has a unique classical solution $u\in C^2(\R_+^d)\cap C(\overline{\R}_+^d)$ satisfying the polynomial growth condition (i.e., there exists a constant $C>0$ and $q\geq 1$ such that $|u(\mathbf{x})|\leq C(1+|\mathbf{x}|^q)$). Moreover, this classical solution $u(\mathbf{x})$ admits the following probabilistic representation given by, for $\mathbf{x}\in\overline{\R}_+^d$,
\begin{align}\label{eq:thm-PRprimalNeumann-elliptic}
u(\mathbf{x}) &= \varphi(\mathbf{x})+\frac{1}{2}\sum_{i\neq j}\Ex\left[\int_0^{\infty} e^{-\rho s+\sum_{k=1}^d c_k L_s^{k,0}}\left(\sum_{k=1}^m\sigma_{ik}\sigma_{jk}\right)\partial_{x_ix_j}^2\varphi(\mathbf{X}_s^{0,\mathbf{x}})ds\right].
\end{align}
Here, the function $\varphi\in C^2(\R_+^d)\cap C(\overline{\R}_+^d)$ admits the following probabilistic representation:
\begin{align*}
\varphi(\mathbf{x}) &=-\sum_{i=1}^d\left\{{\bf 1}_{c_i\neq0}\frac{1}{c_i}\int_0^{\infty} e^{-\rho s}\Ex\left[e^{\sum_{j\neq i}c_j \hat{L}_s^{j,0}} f_i(\hat{\mathbf{X}}_s^{-i,0,\mathbf{x}})\right]d\Ex[e^{c_i \hat{L}_s^{i,0}}]\right.\\
&\qquad\left.+{\bf 1}_{c_i=0}\int_0^{\infty} e^{-\rho s}\Ex\left[e^{\sum_{j\neq i}c_j \hat{L}_s^{j,0}} f_i(\hat{\mathbf{X}}_s^{-i,0,\mathbf{x}})\right]d\Ex[\hat{L}_s^{i,0}]\right\}.
\end{align*}
Our decomposition-homogenization method can be also directly applicable to the Neumann boundary problem, i.e., $c_1=\cdots=c_d=0$ in \eqref{eq:PDE-multidim}. The solution of Neumann boundary problem then admits the following probabilistic representation:
\begin{align}\label{eq:Neumann-integral-local}
u(t,\mathbf{x}) &= \varphi(t,\mathbf{x})+\frac{1}{2}\sum_{i\neq j}\Ex\left[\int_t^{T} e^{-\rho (s-t)}\left(\sum_{k=1}^m\sigma_{ik}\sigma_{jk}\right)\partial_{x_ix_j}^2\varphi(s,\mathbf{X}_s^{t,\mathbf{x}})ds\right]\nonumber\\
&\quad+\Ex\left[e^{-\rho (T-t)} g(\mathbf{X}_T^{t,\mathbf{x}})\right],
\end{align}
where $\varphi \in C^{1,2}([0,T)\times\R_+^d)\cap C([0,T]\times\overline{\R}_+^d)$ admits the form given by
\begin{align}\label{eq:burvarphi}
\varphi(t,\mathbf{x}) &=-\sum_{i=1}^d\int_t^{T} e^{-\rho (s-t)}\Ex[ f_i(s,\hat{\mathbf{X}}_s^{-i,t,\mathbf{x}})]d\Ex[\hat{L}_s^{i,t}].
\end{align}
Similar arguments can be applied to handle the Neumann boundary problem of the elliptic type.
\end{remark}

\section{Proofs of Auxiliary Results}\label{appendix:proof}

This section collects the proofs of all auxiliary results in previous sections.

\begin{proof}[Proof of Lemma~\ref{lem:varphiprobab0}]
It follows from \eqref{eq:varphisol} that
\begin{align}\label{eq:varphisol-inter}
\varphi(t,\mathbf{x}) &=-\sum_{i=1}^d\Ex\left[\int_t^{T} e^{-\rho(s-t)+\sum_{k=1}^d c_k \hat{L}_s^{k,t}}f_i(s,\hat{\mathbf{X}}_s^{-i,t,\mathbf{x}})d\hat{L}_s^{i,t}\right]\nonumber\\
&=-\sum_{i=1}^d \left\{{\bf 1}_{c_i\neq 0}\frac{1}{c_i} \Ex\left[\int_t^{T} e^{-\rho(s-t)+\sum_{j\neq i} c_j \hat{L}_s^{j,t}}f_i(s,\hat{\mathbf{X}}_s^{-i,t,\mathbf{x}})d(e^{c_i\hat{L}_s^{i,t}})\right]\right.\nonumber\\
&\quad\left.+{\bf 1}_{c_i=0} \Ex\left[\int_t^{T} e^{-\rho(s-t)+\sum_{j\neq i} c_j \hat{L}_s^{j,t}}f_i(s,\hat{\mathbf{X}}_s^{-i,t,\mathbf{x}})d\hat{L}_s^{i,t}\right]\right\}.
\end{align}
We only prove the case $c_i\neq 0$, since the other case with $c_i=0$ can be tackled in a similar fashion. Fix $(t,\mathbf{x})\in[0,T]\times\overline{\R}_+^{d}$. Let $s_i:=t+\frac{T-t}{n}i$ for $i=0,1,\ldots,n$ with $n\in\mathbb{N}$. Then, it holds that, for $i=1,\ldots,d$,
\begin{align}\label{eq:sumf}
&\Ex\left[\int_t^{T} e^{-\rho(s-t)+\sum_{j\neq i} c_j \hat{L}_s^{j,t}}f_i(s,\hat{\mathbf{X}}_s^{-i,t,\mathbf{x}})d\left(e^{c_i \hat{L}_s^{i,t}}\right)\right]\\
&\qquad=\Ex\left[\lim_{n\to \infty} \sum_{k=1}^n e^{-\rho (s_{k}-t)+\sum_{j\neq i} c_j \hat{L}_{s_k}^{j,t}}f_i(s_k,\hat{\mathbf{X}}_{s_k}^{-i,t,\mathbf{x}})\left(e^{c_i \hat{L}^{i,t}_{s_k}}-e^{c_i \hat{L}^{i,t}_{s_{k-1}}}\right)\right].\nonumber
\end{align}
In the sequel, let $C>0$ be a generic constant which may be different from line to line. By using Assumption $\bf{(A)}$, we have, for all $i=1,\ldots,d$, $\mathbb{P}$-a.s.
\begin{align*}
\left|f_i(s,\hat{\mathbf{X}}_s^{-i,t,\mathbf{x}})\right|\leq C(1+|\hat{\mathbf{X}}^{t,\mathbf{x}}_s|^q)\leq C\left(1+\sup_{r\in[t,T]}|\hat{\mathbf{X}}^{t,\mathbf{x}}_r|^q\right),\quad \forall s\in[t,T].
\end{align*}
Note that $s\to \hat{L}_s^{i,t}$ is a continuous and non-decreasing process. As a result, $\mathbb{P}$-a.s.
\begin{align*}
&\sum_{k=1}^n e^{-\rho(s_{k}-t)+\sum_{j\neq i} c_j \hat{L}_{s_k}^{j,t}}f_i(s_k,\hat{\mathbf{X}}_{s_k}^{-i,t,\mathbf{x}})\left(e^{c_i \hat{L}^{i,t}_{s_k}}-e^{c_i \hat{L}^{i,t}_{s_{k-1}}}\right)\nonumber\\
&\qquad\qquad\leq  C e^{\sum_{j\neq i} |c_j| \hat{L}_{T}^{j,t}}\left(1+\sup_{r\in[t,T]}|\hat{\mathbf{X}}^{t,\mathbf{x}}_r|^q\right)\sum_{k=1}^n\left(e^{c_i \hat{L}^{i,t}_{s_k}}-e^{c_i \hat{L}^{i,t}_{s_{k-1}}}\right)\nonumber\\
&\qquad\qquad=Ce^{\sum_{j\neq i} |c_j| \hat{L}_T^{j,t}}\left(1+\sup_{r\in[t,T]}|\hat{\mathbf{X}}^{t,\mathbf{x}}_r|^q\right)\hat{L}_T^{i,t}\\
&\qquad\qquad \leq Ce^{3\sum_{j\neq i} |c_j| \hat{L}_T^{j,t}}+C\left(1+\sup_{r\in[t,T]}|\hat{\mathbf{X}}^{t,\mathbf{x}}_s|^q\right)^3+C(\hat{L}_T^{i,t})^3.
\end{align*}
Using the above estimate, we can apply the DCT, the equality \eqref{eq:sumf} and the independence between the reflected state process $\hat{\mathbf{X}}^{-i,t,\mathbf{x}}$ and the local time process $\hat{L}^{i,t}$ to conclude that
\begin{align}\label{eq:m-ex-T}
&\Ex\left[\int_t^{T} e^{-\rho(s-t)+\sum_{j\neq i} c_j \hat{L}_s^{j,t}}f_i(s,\hat{\mathbf{X}}_s^{-i,t,\mathbf{x}})d\left(e^{c_i \hat{L}_s^{i,t}}\right)\right]\nonumber\\
&\quad=\Ex\left[ \lim_{n\to \infty}\sum_{k=1}^n e^{-\rho(s_{k}-t)+\sum_{j\neq i} c_j \hat{L}_{s_k}^{j,t}}f_i(s_k,\hat{\mathbf{X}}_{s_k}^{-i,t,\mathbf{x}})\left(e^{c_i \hat{L}^{i,t}_{s_k}}-e^{c_i \hat{L}^{i,t}_{s_{k-1}}}\right)\right]\nonumber\\
&\quad=\lim_{n\to \infty}\Ex\left[ \sum_{k=1}^n e^{-\rho(s_{k}-t)+\sum_{j\neq i} c_j \hat{L}_{s_k}^{j,t}}f_i(s_k,\hat{\mathbf{X}}_{s_k}^{-i,t,\mathbf{x}})\left(e^{c_i \hat{L}^{i,t}_{s_k}}-e^{c_i \hat{L}^{i,t}_{s_{k-1}}}\right)\right]\nonumber\\
&\quad=\lim_{n\to \infty}\sum_{k=1}^n e^{-\rho(s_{k}-t)} \Ex\left[e^{\sum_{j\neq i} c_j \hat{L}_s^{j,t}}f_i(s_k,\hat{\mathbf{X}}_{s_k}^{-i,t,\mathbf{x}})\right]\left\{\Ex[e^{c_i \hat{L}^{i,t}_{s_k}}]-\Ex[e^{c_i \hat{L}^{i,t}_{s_{k-1}}}]\right\}\nonumber\\
&\quad=\int_t^Te^{-\rho(s-t)}\Ex\left[e^{\sum_{j\neq i} c_j \hat{L}_s^{j,t}}f_i(s,\hat{\mathbf{X}}_{s}^{-i,t,\mathbf{x}})\right]d\Ex[e^{c_i \hat{L}_s^{i,t}}].
\end{align}
Consequently, we conclude from \eqref{eq:varphisol} and \eqref{eq:m-ex-T} that the equality \eqref{eq:varphiinde} holds on $(t,\mathbf{x})\in[0,T]\times\overline{\R}_+^d$. Thus, the proof of the lemma is completed.
\end{proof}

\begin{proof}[Proof of Lemma~\ref{lem:regularity-varphi}]
Without loss of generality, we only provide the proof for the dimension $d=2$ because the proof of the general case with $d>2$ is essentially the same. For $i=1,2$, let us introduce the function $\varphi_i:[0,T]\times\overline{\R}_+^2\to\R$ as follows, for all $(t,x_1,x_2)\in[0,T]\times\overline{\R}^2_+$,
\begin{align}\label{varphi-i}
\varphi_i(t,x_1,x_2):=-\Ex\left[\int_t^{T} e^{-\rho(s-t)+\sum_{k=1}^2 c_k \hat{L}_s^{k,t}}f_i(s,\hat{\mathbf{X}}_s^{-i,t,\mathbf{x}})d\hat{L}_s^{i,t}\right].
\end{align}
Then, it follows from Lemma \ref{lem:varphiprobab0} that
\begin{align}\label{varphi-i-independent}
\varphi_i(t,x_1,x_2)=
\begin{cases}
\displaystyle -\frac{1}{c_i}\int_t^{T} e^{-\rho(s-t)}\Ex\left[e^{\sum_{j\neq i} c_j \hat{L}_s^{j,t}}f_i(s,\hat{\mathbf{X}}_s^{-i,t,\mathbf{x}})\right]d\Ex[e^{c_i\hat{L}_s^{i,t}}], &c_i\neq 0.\\[1em]
\displaystyle -\int_t^{T} e^{-\rho(s-t)}\Ex\left[e^{\sum_{j\neq i} c_j \hat{L}_s^{j,t}}f_i(s,\hat{\mathbf{X}}_s^{-i,t,\mathbf{x}})\right]d\Ex[\hat{L}_s^{i,t}], &c_i=0.
\end{cases}
\end{align}
In the sequel, we only handle the case for the function $\varphi_1$ because the case for $\varphi_2$ is similar. It suffices to prove that $\varphi_1\in C^{1,2}([0,T)\times\R_+^2)\cap C([0,T]\times\overline{\R}_+^2)$ satisfies that
\begin{align}
\begin{cases}
\displaystyle \partial_{x_1} \varphi_1(t,x_1,x_2)=\Ex\left[e^{-\rho (\tau_{x_1}^1-t)+c_2 \hat{L}_{\tau_{x_1}^1}^{2,t}}f_1(\tau_{x_1}^1,\hat{X}_{\tau_{x_1}^1}^{2,t,x_2}){\bf 1}_{\tau_{x_1}^1<T}\right]\\[1.2em]
\displaystyle\qquad\qquad\qquad\qquad+c_1 \mathbb{E}\left[\int_{\tau_{x_1}^1\wedge T}^{T}  e^{-\rho (s-t)+c_1 \hat{L}_s^{1,t}+c_2 \hat{L}_s^{2,t}}f_1(s,\hat{X}_s^{2,t,x_2}) d\hat{L}_s^{1,t}\right],\\[1.4em]
\displaystyle \partial_{x_2} \varphi_1(t,x_1,x_2)=-\mathbb{E}\left[\int_t^{\tau^2_{x_2}\wedge T }e^{-\rho (s-t)+c_1 \hat{L}_s^{1,t}} \partial_x f_1(s,\hat{X}_s^{2,t,x_2}) d \hat{L}_s^{1,t}\right]\\[1.2em]
\displaystyle\qquad\qquad\qquad\qquad+c_2\mathbb{E}\left[\int_{\tau^2_{x_2}\wedge T }^T e^{-\rho (s-t)+c_1 \hat{L}_s^{1,t}+c_2 \hat{L}_s^{2,t}} f_1(s,\hat{X}_s^{2,t,x_2}) d \hat{L}_s^{1,t}\right].
\end{cases}
\end{align}
First of all, we consider the case with arbitrary $x_2'>x_2 \geq 0$ and fixed $(t,x_1)\in[0,T]\times \overline{\R}_+$. It follows from \eqref{varphi-i} that
\begin{align*}
&\frac{\varphi_1(t,x_1,x_2')-\varphi_1(t,x_1,x_2)}{x_2'-x_2}\nonumber\\
&\qquad=-\mathbb{E}\left[\int_t^{T} e^{-\rho (s-t)+c_1 \hat{L}_s^{1,t}} \frac{e^{c_2 \hat{L}_s^{2',t}}f_1(s,\hat{X}_s^{2,t,x_2'})-e^{c_2 \hat{L}_s^{2,t}}f_1(s,\hat{X}_s^{2,t,x_2})}{x_2'-x_2}d \hat{L}_s^{1,t}\right]\\
&\qquad=-\mathbb{E}\left[\int_t^{T} e^{-\rho (s-t)+c_1 \hat{L}_s^{1,t}} \frac{e^{c_2 \hat{L}_s^{2',t}}\left(f_1(s,\hat{X}_s^{2,t,x_2'})-f_1(s,\hat{X}_s^{2,t,x_2})\right)}{x_2'-x_2}d \hat{L}_s^{1,t}\right]\\
&\qquad\quad-\mathbb{E}\left[\int_t^{T} e^{-\rho (s-t)+c_1 \hat{L}_s^{1,t}} \frac{f_1(s,\hat{X}_s^{2,t,x_2})\left(e^{c_2 \hat{L}_s^{2',t}}-e^{c_2 \hat{L}_s^{2,t}}\right)}{x_2'-x_2}d \hat{L}_s^{1,t}\right].
\end{align*}
A direct calculation yields that, for all $s \in [t,T]$,
\begin{align*}
&\lim _{x_2' \downarrow x_2} \frac{f_1(s,\hat{X}_s^{2,t,x_2'})-f_1(s,\hat{X}_s^{2,t,x_2})}{x_2'-x_2}\\
&\quad=\left\{\begin{array}{cc}
\displaystyle \partial_x f_1(s,\hat{X}_s^{2,t,x_2}), & \max_{\ell\in[t,s]}\left\{ -\mu_2(\ell-t) -\sum_{k=1}^m \sigma_{2k}(W_{\ell}^{k,2}-W_t^{k,2})\right\} \leq x_2, \\[0.8em]
\displaystyle 0, & \max_{\ell\in[t,s]}\left\{-\mu_2(\ell-t) -\sum_{k=1}^m \sigma_{2k}(W_{\ell}^{k,2}-W_t^{k,2})\right\} >x_2.
\end{array}\right.
\end{align*}
By Assumption {\bf(A)}, for any $\epsilon>0$ and $x_2'\in(x_2,x_2+\epsilon)$, there exists some constant $C>0$  and $q\geq 1$ such that
\begin{align*}
\left|\frac{f_1(s,\hat{X}_s^{2,t,x_2'})-f_1(s,\hat{X}_s^{2,t,x_2})}{x_2'-x_2}\right|&\leq C\left[1+\left|\hat{X}_s^{2,t,x_2}\right|^q+\left|\hat{X}_s^{2,t,x_2'}\right|^q\right]\\
&\leq C\left[1+\left|\hat{X}_s^{2,t,x_2}\right|^q+\left|\hat{X}_s^{2,t,x_2+\epsilon}\right|^q\right],
\end{align*}
where the validity of the last inequality is due to \eqref{eq:RSDEi}. Denote by $\hat{L}^{2',t}=(\hat{L}^{2',t}_s)_{s\in[t,T]}$ the local time process of $\hat{X}^{2,t,x_2'}=(\hat{X}_s^{2,t,x_2'})_{s\in[t,T]}$ with initial state $\hat{X}_t^{2,t,x_2'}=x_2'$ at initial time $t$.
It then follows from the DCT and the fact that $\hat{L}_s^{2,t}=0$ on $[t,\tau_{x_2}^2]$ that
\begin{align}\label{eq:varphi-x2-1}
&\lim _{x_2'\in(x_2,x_2+\epsilon), ~x_2' \downarrow x_2} \int_t^{T} \mathbb{E}\left[e^{-\rho (s-t)+c_1 \hat{L}_s^{1,t}+c_2 \hat{L}_s^{2',t}} \frac{f_1(s,\hat{X}_s^{2,t,x_2'})-f_1(s,\hat{X}_s^{2,t,x_2})}{x_2'-x_2}d \hat{L}_s^{1,t}\right]\nonumber \\
&\qquad\qquad=\mathbb{E}\bigg[ \int_t^{T}e^{-\rho (s-t)+c_1 \hat{L}_s^{1,t}+c_2 \hat{L}_s^{2,t}} \partial_xf_1(s,\hat{X}_s^{2,t,x_2})\nonumber\\
&\qquad\qquad\qquad\quad\times{\bf1}_{\max_{\ell\in[t,s]}\{ -\mu_2 (\ell-t) -\sum_{k=1}^m \sigma_{2k}(W_{\ell}^{k,2}-W_t^{k,2})\} \leq x_2} d\hat{L}_s^{1,t}\bigg] \nonumber\\
&\qquad\qquad=\mathbb{E}\left[\int_t^{\tau^2_{x_2}\wedge T }e^{-\rho (s-t)+c_1 \hat{L}_s^{1,t}+c_2 \hat{L}_s^{2,t}} \partial_x f_1(s,\hat{X}_s^{2,t,x_2}) d \hat{L}_s^{1,t}\right]\nonumber\\
&\qquad\qquad=\mathbb{E}\left[\int_t^{\tau^2_{x_2}\wedge T }e^{-\rho (s-t)+c_1 \hat{L}_s^{1,t}} \partial_x f_1(s,\hat{X}_s^{2,t,x_2}) d \hat{L}_s^{1,t}\right],
\end{align}
where $\tau_{x_2}^2$ is a stopping time which is defined by
\begin{align}\label{tau-j}
\tau_{x_2}^2:=\inf\left\{s\geq t;~-\mu_2 (s-t) -\sum_{k=1}^m \sigma_{2k}(W_s^{k,2}-W_t^{k,2})=x_2\right\}.
\end{align}
On the other hand, in lieu of \eqref{eq:RSDEi}, it follows that, for all $s\in[t,T]$,
\begin{align*}
\begin{cases}
\displaystyle \hat{L}_s^{2,t}=x_2 \vee\left\{\max_{\ell\in[t,s]}\left(-\mu_2 (\ell-t) -\sum_{k=1}^m \sigma_{2k}(W_{\ell}^{k,2}-W_t^{k,2})\right)\right\}-x_2,\\[1.5em]
\displaystyle \hat{L}_s^{2',t}=x_2' \vee\left\{\max_{\ell\in[t,s]}\left(-\mu_2 (\ell-t) -\sum_{k=1}^m \sigma_{2k}(W_{\ell}^{k,2}-W_t^{k,2})\right)\right\}-x_2'.
\end{cases}
\end{align*}
This implies that
\begin{align*}
&\lim _{x_2' \downarrow x_2} \frac{e^{c_2 \hat{L}_s^{2',t}}-e^{c_2 \hat{L}_s^{2,t}}}{x_2'-x_2}\\
&\quad=\begin{cases}
\displaystyle ~~~0, & \max_{\ell\in[t,s]}\left\{ -\mu_2(\ell-t) -\sum_{k=1}^m \sigma_{2k}(W_{\ell}^{k,2}-W_t^{k,2})\right\} \leq x_2, \\[1.2em]
\displaystyle -c_2 e^{c_2 \hat{L}_s^{2,t}}, & \max_{\ell\in[t,s]}\left\{-\mu_2(\ell-t) -\sum_{k=1}^m \sigma_{2k}(W_{\ell}^{k,2}-W_t^{k,2})\right\} >x_2.
\end{cases}
\end{align*}
Then, we obtain from the DCT that
\begin{align}\label{eq:varphi-x2-2}
&\lim _{x_2'\in(x_2,x_2+\epsilon), ~x_2' \downarrow x_2} \mathbb{E}\left[\int_t^{T} e^{-\rho (s-t)+c_1 \hat{L}_s^{1,t}} \frac{f_1(s,\hat{X}_s^{2,t,x_2})(e^{c_2 \hat{L}_s^{2',t}}-e^{c_2 \hat{L}_s^{2,t}})}{x_2'-x_2}d \hat{L}_s^{1,t}\right]\nonumber \\
&\qquad\quad=-c_2 \mathbb{E}\bigg[ \int_t^{T}e^{-\rho (s-t)+c_1 \hat{L}_s^{1,t}+c_2 \hat{L}_s^{2,t}} f_1(s,\hat{X}_s^{2,t,x_2})\nonumber\\
&\qquad\qquad\qquad\times{\bf1}_{\max_{\ell\in[t,s]}\{ -\mu_2 (\ell-t) -\sum_{k=1}^m \sigma_{2k}(W_{\ell}^{k,2}-W_t^{k,2})\}> x_2} d\hat{L}_s^{1,t}\bigg] \nonumber\\
&\qquad\quad=-c_2\mathbb{E}\left[\int_{\tau^2_{x_2}\wedge T }^T e^{-\rho (s-t)+c_1 \hat{L}_s^{1,t}+c_2 \hat{L}_s^{2,t}}f_1(s,\hat{X}_s^{2,t,x_2}) d \hat{L}_s^{1,t}\right].
\end{align}
For the case $x_2>x_2' \geq 0$, similar to the computation for \eqref{eq:varphi-x2-1} and \eqref{eq:varphi-x2-2}, we can show that
\begin{align}\label{eq:varphi-x2-3}
\lim _{x_2' \uparrow x_2} \frac{\varphi_1(t,x_1,x_2')-\varphi_1(t,x_1,x_2)}{x_2'-x_2}=\lim _{x_2' \downarrow x_2} \frac{\varphi_1(t,x_1,x_2')-\varphi_1(t,x_1,x_2)}{x_2'-x_2}.
\end{align}
Thus, we can deduce from \eqref{eq:varphi-x2-1}, \eqref{eq:varphi-x2-2} and \eqref{eq:varphi-x2-3} that
\begin{align}\label{varphi-x}
\partial_{x_2} \varphi_1(t,x_1,x_2)&=-\mathbb{E}\left[\int_t^{\tau^2_{x_2}\wedge T }e^{-\rho (s-t)+c_1 \hat{L}_s^{1,t}} \partial_x f_1(s,\hat{X}_s^{2,t,x_2}) d \hat{L}_s^{1,t}\right]\nonumber\\
&\quad+c_2\mathbb{E}\left[\int_{\tau^2_{x_2}\wedge T }^T e^{-\rho (s-t)+c_1 \hat{L}_s^{1,t}+c_2 \hat{L}_s^{2,t}} f_1(s,\hat{X}_s^{2,t,x_2}) d \hat{L}_s^{1,t}\right].
\end{align}
It further holds that
\begin{align*}
&\mathbb{E}\left[\int_{\tau^2_{x_2}\wedge T }^T e^{-\rho (s-t)+c_1 \hat{L}_s^{1,t}+c_2 \hat{L}_s^{2,t}} f_1(s,\hat{X}_s^{2,t,x_2}) d \hat{L}_s^{1,t}\right]\\
&\qquad=\mathbb{E}\left[\int_{t}^T e^{-\rho (s-t)+c_1 \hat{L}_s^{1,t}+c_2 \hat{L}_s^{2,t}} f_1(s,\hat{X}_s^{2,t,x_2}) d \hat{L}_s^{1,t}\right]\\
&\qquad\quad-\mathbb{E}\left[\int_t^{\tau^2_{x_2}\wedge T } e^{-\rho (s-t)+c_1 \hat{L}_s^{1,t}+c_2 \hat{L}_s^{2,t}} f_1(s,\hat{X}_s^{2,t,x_2}) d \hat{L}_s^{1,t}\right]\\
&\qquad=-\varphi_1(t,x_1,x_2)-\mathbb{E}\left[\int_t^{\tau^2_{x_2}\wedge T } e^{-\rho (s-t)+c_1 \hat{L}_s^{1,t}} f_1(s,\hat{X}_s^{2,t,x_2}) d \hat{L}_s^{1,t}\right].
\end{align*}
 Thus, we can conclude that
\begin{align}\label{varphi-x-2}
&\partial_{x_2} \varphi_1(t,x_1,x_2)\nonumber\\
&\quad=-\mathbb{E}\left[\int_t^{\tau^2_{x_2}\wedge T }e^{-\rho (s-t)+c_1 \hat{L}_s^{1,t}} \partial_x f_1(s,\hat{X}_s^{2,t,x_2}) d \hat{L}_s^{1,t}\right]\nonumber\\
&\qquad-c_2\varphi_1(t,x_1,x_2)-c_2\mathbb{E}\left[\int_t^{\tau^2_{x_2}\wedge T } e^{-\rho (s-t)+c_1 \hat{L}_s^{1,t}} f_1(s,\hat{X}_s^{2,t,x_2}) d \hat{L}_s^{1,t}\right].\nonumber\\
&\quad=\mathbb{E}\left[\int_t^{\tau^2_{x_2}\wedge T }e^{-\rho (s-t)+c_1 \hat{L}_s^{1,t}}F(s,\hat{X}_s^{2,t,x_2})d \hat{L}_s^{1,t}\right]-c_2\varphi_1(t,x_1,x_2),
\end{align}
where the real-valued function $F(t,x_1):=-c_2 f_1(t,x_1)-\partial_{x_1} f_1(t,x_1)$ for $(t,x_1)\in[0,T]\times \overline{\R}_+$.

For any $x_2, x_2^{(n)} \geq 0$ with $x_2^{(n)} \rightarrow x_2$ as $n \to\infty$ and $(t,x_1)\in[0,T]\times\overline{\R}_+$, we have from \eqref{varphi-x-2} that, for all $n \geq 1$,
\begin{align*}
\Delta_n&:=\frac{\partial_{x_2} \varphi_1(t,x_1,x_2^{(n)})-\partial_{x_2} \varphi_1(t,x_1,x_2)}{x_2^{(n)}-x_2}\\
&= \Ex\left[\frac{1}{x_2^{(n)}-x_2}\int_{\tau_{0}}^{\tau_{n}} e^{-\rho (s-t)+c_1 \hat{L}_s^{1,t}}F(s,\hat{X}_s^{2,t,x_2})d\hat{L}_s^{1,t}\right]\nonumber\\
&\qquad\quad+\Ex\left[\frac{1}{x_2^{(n)}-x_2}\int_{t}^{\tau_{0}}e^{-\rho (s-t)+c_1 \hat{L}_s^{1,t}}\left( F(s,\hat{X}_s^{2,t,x_2^{(n)}})-F(s,\hat{X}_s^{2,t,x_2})\right)d\hat{L}_s^{1,t}\right]\nonumber
\end{align*}
\begin{align}\label{eq:Deltan}
&\quad+ \Ex\left[\frac{1}{x_2^{(n)}-x_2}\int_{\tau_{0}}^{\tau_{n}} e^{-\rho (s-t)+c_1 \hat{L}_s^{1,t}}\left( F(s,\hat{X}_s^{2,t,x_2^{(n)}})-F(s,\hat{X}_s^{2,t,x_2})\right)d\hat{L}_s^{1,t}\right]\nonumber\\
&\quad-c_2\frac{\varphi_1(t,x_1,x_2^{(n)})-\varphi_1(t,x_1,x_2)}{x_2^{(n)}-x_2}\nonumber\\
&:=\Delta_n^{(1)}+\Delta_n^{(2)}+\Delta_n^{(3)}+\Delta_n^{(4)},
\end{align}
where, for simplicity, we set $\tau_n:=\tau^2_{x_2^{(n)}}\wedge T$ and $\tau_0:=\tau^2_{x_2}\wedge T$.

In order to handle the term $\Delta_n^{(1)}$, we first focus on the case with $x_2^{(n)} \downarrow x_2$ as $n \to \infty$. Let us define the drifted-Brownian motion $\tilde{W}^{i,t}=(\tilde{W}_s^{i,t})_{s\in[t,T]}$ given by, for $i=1,2$,
\begin{align}\label{eq:tildeW}
\tilde{W}_s^{i,t}:=-\mu_i (s-t) -\sum_{k=1}^m \sigma_{ik}(W_s^{k,i}-W_t^{k,i}),\quad \forall s\in[t,T].
\end{align}
It follows from \eqref{eq:RSDEi} that
\begin{align}\label{eq:Xi-max}
\hat{X}_s^{i,t,x_i}&=x_i +  \int_t^s \mu_i dr + \sum_{k=1}^m \int_t^T \sigma_{ik}dW_s^{k,i} + \hat{L}_s^{i,t}\\
&=x_i-\tilde{W}_s^{i,t}+\left(\sup_{\ell\in[t,s]}\tilde{W}_{\ell}^{i,t}-x_i\right)^+,\quad \forall s\in[t,T].\nonumber
\end{align}
 For $s\in[t,T]$, let $\phi_i(t,s,r,y)$ be the joint probability density of the two-dimensional random variable $(\tilde{W}_s^{i,t},\max_{\ell\in[t,s]}\tilde{W}^{i,t}_{\ell})$ (c.f. \citealt{Harrison85}). That is, for all $(s,r,y)\in[t,T]\times\R^2$,
\begin{align}\label{eq:phii}
\phi_i(t,s,r,y)=\frac{2(2 y-r)}{\tilde{\sigma}_i^2\sqrt{2\tilde{\sigma}_i^2 \pi (s-t)^3}}\exp\left[\frac{\mu_i}{\tilde{\sigma}_i} x-\frac{1}{2} \mu_i^2 (s-t)-\frac{(2 y-r)^2}{2\tilde{\sigma}_i^2 (s-t)}\right],
\end{align}
where the constant $\tilde{\sigma}_i:=\sqrt{\sum_{k=1}^m\sigma_{ik}^2}$ for $i=1,2$. Consequently, for the case $c_1\neq 0$, we have
\begin{align}\label{eq:delat-n1-1}
&\Ex\left[\frac{1}{x_2^{(n)}-x_2}\int_{\tau_{0}}^{\tau_{n}}e^{-\rho (s-t)+c_1 \hat{L}_s^{1,t}} F(s,\hat{X}_s^{2,t,x_2})d\hat{L}_s^{1,t}\right]\\
&\quad=\frac{1}{c_1}\Ex\left[\frac{1}{x_2^{(n)}-x_2}\int_{\tau_{0}}^{\tau_{n}}e^{-\rho (s-t)}  F(s,\hat{X}_s^{2,t,x_2})\mathbf{1}_{ \{\tau_{0}<s\leq \tau_{n}\}}de^{c_1 \hat{L}_s^{1,t}}\right]\nonumber\\
&\quad=\frac{1}{c_1}\int_{t}^{T}e^{-\rho (s-t)}\Ex\left[\frac{F\left(s, x_2-\tilde{W}_s^{2,t}+(\sup_{q\in[t,s]}\tilde{W}_q^{2,t}-x_2)^+\right)}{x_2^{(n)}-x_2}\right.\nonumber\\
&\qquad\qquad\qquad\left.\times \mathbf{1}_{\{x_2<\max _{q \in[t, s]}\tilde{W}_q^{2,t} \leq x_2^{(n)}\}}\right]d\Ex[e^{c_1\hat{L}_s^{1,t}}]\nonumber\\
&\quad=\frac{1}{c_1}\frac{1}{x_2^{(n)}-x_2} \int_t^T\int_{x_2}^{x_2^{(n)}}\int_{-\infty}^{y} e^{-\rho (s-t)} F(s,y-r) \phi_2(t,s,r,y)drdy d\Ex[e^{c_1 \hat{L}_s^{1,t}}].\nonumber
\end{align}
For $0\leq t\leq s\leq T$ and $y\in\overline{\R}_+$, set $g(t,s,y):=\int_{-\infty}^{y}  \partial_xf_1(s,y-r)\phi_2(t,s,r,y)dr$.  Then, using the continuity of $y\to g(t,s,y)$, one has
\begin{align}\label{eq:delat-n1-2}
\lim_{n\to \infty} \frac{1}{x_2^{(n)}-x_2}\int_{x_2}^{x_2^{(n)}} g(t,s,y)dy=g(t,s,x_2).
\end{align}
It then follows from \eqref{eq:delat-n1-1}, \eqref{eq:delat-n1-2} and the DCT that
\begin{align*}
&\lim_{n\to \infty} \Ex\left[\frac{1}{x_2^{(n)}-x_2}\int_{\tau_{0}}^{\tau_{n}}e^{-\rho (s-t)+c_1 \hat{L}_s^{1,t}} F(s,\hat{X}_s^{2,t,x_2})d\hat{L}_s^{1,t}\right]\\
&\qquad=\frac{1}{c_1}\int_t^T\left(\int_{-\infty}^{x_2} e^{-\rho (s-t)} F(s,x_2-r) \phi_2(t,s,r,x_2)dr\right) d\Ex[e^{c_1 \hat{L}_s^{1,t}}].\nonumber
\end{align*}

For the case when $x_2>0$ and $x_2^{(n)}>0$ with $x_2^{(n)}\uparrow x_2$ as $n\to \infty$, using a similar argument as above, we can derive that
\begin{align*}
&\lim_{n\to \infty}\Delta_n^{(1)} =\frac{1}{c_1} \int_t^T\left(\int_{-\infty}^{x_2} e^{-\rho (s-t)} F(s,x_2-r) \phi_2(t,s,r,x_2)dr\right) d\Ex[e^{c_1 \hat{L}_s^{1,t}}].
\end{align*}
If $c_1=0$, we can also obtain that
\begin{align*}
&\lim_{n\to \infty}\Delta_n^{(1)} =\int_t^T\left(\int_{-\infty}^{x_2} e^{-\rho (s-t)} F(s,x_2-r) \phi_2(t,s,r,x_2)dr\right) d\Ex[\hat{L}_s^{1,t}].
\end{align*}
In a similar fashion as in the derivation of \eqref{eq:varphi-x2-1}, we also have
\begin{align*}
\lim_{n\to \infty } \Delta_n^{(2)}
&=\lim_{n\to \infty }\Ex\left[\frac{1}{x_2^{(n)}-x_2}\int_{t}^{\tau_{0}}e^{-\rho (s-t)+c_1 \hat{L}_s^{1,t}}\left( F(s,\hat{X}_s^{2,t,x_2^{(n)}})- F(s,\hat{X}_s^{2,t,x_2})\right)d\hat{L}_s^{1,t}\right]\\
&=\Ex\left[\int_{t}^{\tau_{0}}e^{-\rho (s-t)+c_1 \hat{L}_s^{1,t}}\partial_{x}F(s,\hat{X}_s^{2,t,x_2})d\hat{L}_s^{1,t}\right].
\end{align*}
For the term $\Delta_n^{(3)}$, we can also conclude that
\begin{align*}
\left|\Delta_n^{(3)}\right|
&=\left|\Ex\left[\frac{1}{x_2^{(n)}-x_2}\int_{\tau_{0}}^{\tau_{n}} e^{-\rho (s-t)+c_1 \hat{L}_s^{1,t}}\left(F(s,\hat{X}_s^{2,t,x_2^{(n)}})- F(s,\hat{X}_s^{2,t,x_2})\right)d\hat{L}_s^{1,t}\right]\right|\\
&\leq \Ex\left[\frac{1}{x_2^{(n)}-x_2}\sup_{s\in [t,T]}\left|  F(s,\hat{X}_s^{2,t,x^{(n)}_2})-F(s,\hat{X}_s^{2,t,x_2})\right|\int_{\tau_{0}}^{\tau_{n}} e^{c_1 \hat{L}_s^{1,t}}d\hat{L}_s^{1,t}\right].
\end{align*}
By applying Assumption {\bf(A)}, for $\epsilon>0$ and $x_2'\in(x_2,x_2+\epsilon)$, there exist some constants $C>0$ and $q\geq 1$ such that
\begin{align*}
\left|\frac{F(s,\hat{X}_s^{2,t,x_2'})-F(s,\hat{X}_s^{2,t,x_2})}{x_2'-x_2}\right|&\leq C\left[1+\left|\hat{X}_s^{2,t,x_2}\right|^q+\left|\hat{X}_s^{2,t,x_2'}\right|^q\right]\\
&\leq C\left[1+\left|\hat{X}_s^{2,t,x_2}\right|^q+\left|\hat{X}_s^{2,t,x_2+\epsilon}\right|^q\right].
\end{align*}
Using the result above and the fact that $\tau_{n} \to \tau_{0}$, a.s., as $n\to \infty$, we deduce from the DCT  that $\lim_{n\to \infty}|\Delta_n^{(3)}|=0$.
At last, we can also see that
\begin{align*}
\lim_{n\to \infty}\Delta_n^{(4)}=-c_2\lim_{n\to \infty}\frac{\varphi_1(t,x_1,x_2^{(n)})-\varphi_1(t,x_1,x_2)}{x_2^{(n)}-x_2}=-c_2 \partial_{x_2} \varphi_1(t,x_1,x_2).
\end{align*}
Putting all the pieces together, we conclude that
\begin{align}\label{eq:m-rr1}
&\partial_{x_2x_2}\varphi_1(t,x_1,x_2)={\bf 1}_{c_1\neq 0}\frac{1}{c_1}\int_t^T\left(\int_{-\infty}^{x_2} e^{-\rho (s-t)} F(s,x_2-r) \phi_2(t,s,r,x_2)dr\right) d\Ex[e^{c_1 \hat{L}_s^{1,t}}]\nonumber\\
&\qquad\qquad+{\bf 1}_{c_1=0}\int_t^T\left(\int_{-\infty}^{x_2} e^{-\rho (s-t)} F(s,x_2-r) \phi_2(t,s,r,x_2)dr\right) d\Ex[\hat{L}_s^{1,t}]\nonumber\\
&\qquad\qquad+\Ex\left[\int_{t}^{\tau_{0}}e^{-\rho (s-t)+c_1 \hat{L}_s^{1,t}}\left( \partial_{x}F(s,\hat{X}_s^{2,t,x_2})\right)d\hat{L}_s^{1,t}\right]-c_2 \partial_{x_2} \varphi_1(t,x_1,x_2).
\end{align}

We next derive the representation of the partial derivative $\partial_{x_1}\varphi_1(t,x_1,x_2)$. For any $x_1'>x_1 \geq 0$, and fixed $(t,x_2)\in[0,T]\times \overline{\R}_+$, it follows from \eqref{varphi-i} that
\begin{align*}
&\frac{\varphi_1(t,x_1',x_2)- \varphi_1(t,x_1,x_2)}{x_1'-x_1}\\
&=-\mathbb{E}\left[\int_t^{T}  e^{-\rho (s-t)+c_1 \hat{L}_s^{1',t}+c_2 \hat{L}_s^{2,t}}f_1(s,\hat{X}_s^{2,t,x_2}) d\left(\frac{\hat{L}_s^{1',t}-\hat{L}_s^{1,t}}{x_1'-x_1}\right)\right]\\
&~-\mathbb{E}\left[\int_t^{T}  e^{-\rho (s-t)+c_2 \hat{L}_s^{2,t}}f_1(s,\hat{X}_s^{2,t,x_2})\left(\frac{e^{c_1 \hat{L}_s^{1',t}}-e^{c_1\hat{L}_s^{1,t}}}{x_1'-x_1}\right) d\hat{L}_s^{1,t}\right]
\end{align*}
In lieu of \eqref{eq:RSDEi}, it follows that, for all $s\in[t,T]$,
\begin{align*}
\begin{cases}
\displaystyle \hat{L}_s^{1,t}=x_1 \vee\left\{\max_{\ell\in[t,s]}\left(-\mu_1 (\ell-t) -\sum_{k=1}^m \sigma_{1k}(W_{\ell}^{k,1}-W_t^{k,1})\right)\right\}-x_1,\\[1.4em]
\displaystyle \hat{L}_s^{1',t}=x_1' \vee\left\{\max_{\ell\in[t,s]}\left(-\mu_1 (\ell-t) -\sum_{k=1}^m \sigma_{1k}(W_{\ell}^{k,1}-W_t^{k,1})\right)\right\}-x_1'.
\end{cases}
\end{align*}
 A direct calculation yields that, for all $s\in[t,T]$,
\begin{align*}
&\frac{\hat{L}_s^{1',t}-\hat{L}_s^{1,t}}{x_1'-x_1}=
\begin{cases}
\displaystyle ~~~~0, & ~~~t\leq s< \tau_{x_1}^1\wedge T, \\[0.4em]
\displaystyle-\frac{\hat{L}_s^{1,t}}{x_1'-x_1},  & \tau_{x_1}^1\wedge T\leq s<\tau_{x_1'}^1\wedge T,\\[0.7em]
\displaystyle~~~-1, & ~~~ \tau_{x_1'}^1\wedge T\leq s\leq T.
\end{cases}
\end{align*}
Then,  it holds that
\begin{align*}
&\mathbb{E}\left[\int_t^{T}e^{-\rho (s-t)+c_1 \hat{L}_s^{1',t}+c_2 \hat{L}_s^{2,t}}  f_1(s,\hat{X}_s^{2,t,x_2}) d\left(\frac{\hat{L}_s^{1',t}-\hat{L}_s^{1,t}}{x_1'-x_1}\right)\right]\\
&\qquad=-\mathbb{E}\left[\int_{\tau_{x_1}^1\wedge T}^{\tau_{x_1'}^1\wedge T}e^{-\rho (s-t)+c_1 \hat{L}_s^{1',t}+c_2 \hat{L}_s^{2,t}} f_1(s,\hat{X}_s^{2,t,x_2}) d\left(\frac{\hat{L}_s^{1,t}}{x_1'-x_1}\right)\right]   \nonumber\\
&\qquad=-\mathbb{E}\left[\int_{\tau_{x_1}^1\wedge T}^{\tau_{x_1'}^1\wedge T}\left(e^{-\rho (s-t)+c_1 \hat{L}_s^{1',t}+c_2 \hat{L}_s^{2,t}}f_1(s,\hat{X}_s^{2,t,x_2})\right.\right.\\
&\qquad\qquad\qquad\quad\left.\left.-e^{-\rho (\tau_{x_1}^1-t)+c_1 \hat{L}_{\tau_{x_1}^1}^{1',t}+c_2 \hat{L}_{\tau_{x_1}^1}^{2,t}}f_1(\tau_{x_1}^1,\hat{X}_{\tau_{x_1}^1}^{2,t,x_2})\right) d\left(\frac{\hat{L}_s^{1,t}}{x_1'-x_1}\right)\right]\\
&\qquad\quad-\Ex\left[\int_{\tau_{x_1}^1\wedge T}^{\tau_{x_1'}^1\wedge T}e^{-\rho (\tau_{x_1}^1-t)+c_1 \hat{L}_{\tau_{x_1}^1}^{1',t}+c_2 \hat{L}_{\tau_{x_1}^1}^{2,t}}f_1(\tau_{x_1}^1,\hat{X}_{\tau_{x_1}^1}^{2,t,x_2})d\left(\frac{\hat{L}_s^{1,t}}{x_1'-x_1}\right)\right].
\end{align*}
Note that, as $x_1' \downarrow x_1$, we have
\begin{align*}
&\left|\mathbb{E}\left[\int_{\tau_{x_1}^1\wedge T}^{\tau_{x_1'}^1\wedge T}\left(e^{-\rho (s-t)+c_1 \hat{L}_s^{1',t}+c_2 \hat{L}_s^{2,t}}f_1(s,\hat{X}_s^{2,t,x_2})\right.\right.\right.\\
&\qquad\qquad\quad\left.\left.\left.-e^{-\rho (\tau_{x_1}^1-t)+c_1 \hat{L}_{\tau_{x_1}^1}^{1',t}+c_2 \hat{L}_{\tau_{x_1}^1}^{2,t}}f_1(\tau_{x_1}^1,\hat{X}_{\tau_{x_1}^1}^{2,t,x_2})\right) d\left(\frac{\hat{L}_s^{1,t}}{x_1'-x_1}\right)\right]\right|\\
&\leq \Ex\left[\sup_{s\in [\tau_{x_1}^1\wedge T,\tau_{x_1'}^1\wedge T]}\left|e^{-\rho (s-t)+c_1 \hat{L}_s^{1',t}+c_2 \hat{L}_s^{2,t}}f_1(s,\hat{X}_s^{2,t,x_2})\right.\right.\\
&\left.\left.\qquad\qquad\qquad-e^{-\rho (\tau_{x_1}^1-t)+c_1 \hat{L}_{\tau_{x_1}^1}^{1',t}+c_2 \hat{L}_{\tau_{x_1}^1}^{2,t}}f_1(\tau_{x_1}^1,\hat{X}_{\tau_{x_1}^1}^{2,t,x_2})\right|\right]\to 0.
\end{align*}
 On the other hand, as $x_1' \downarrow x_1$, we can also obtain 
\begin{align*}
&\lim _{x_1' \downarrow x_1}\int_{\tau_{x_1}^1\wedge T}^{\tau_{x_1'}^1\wedge T}e^{-\rho (\tau_{x_1}^1-t)+c_1 \hat{L}_{\tau_{x_1}^1}^{1',t}+c_2 \hat{L}_{\tau_{x_1}^1}^{2,t}}f_1(\tau_{x_1}^1,\hat{X}_{\tau_{x_1}^1}^{2,t,x_2})d\left(\frac{\hat{L}_s^{1,t}}{x_1'-x_1}\right)\\
&\quad = e^{-\rho (\tau_{x_1}^1-t)+c_1 \hat{L}_{\tau_{x_1}^1}^{1,t}+c_2 \hat{L}_{\tau_{x_1}^1}^{2,t}}f_1(\tau_{x_1}^1,\hat{X}_{\tau_{x_1}^1}^{2,t,x_2})\times\lim _{x_1' \downarrow x_1}{\bf 1}_{\tau_{x_1}^1<T}\int_{\tau_{x_1}^1\wedge T}^{\tau_{x_1'}^1\wedge T}d\left(\frac{\hat{L}_s^{1,t}}{x_1'-x_1}\right)\\
&\quad=e^{-\rho (\tau_{x_1}^1-t)+c_1 \hat{L}_{\tau_{x_1}^1}^{1,t}+c_2 \hat{L}_{\tau_{x_1}^1}^{2,t}}f_1(\tau_{x_1}^1,\hat{X}_{\tau_{x_1}^1}^{2,t,x_2})\times \lim _{x_1' \downarrow x_1} {\bf 1}_{\tau_{x_1}^1<T}\int_{\tau_{x_1}^1}^{\tau_{x_1'}^1}d\left(\frac{\hat{L}_s^{1,t}}{x_1'-x_1}\right)\\
&\quad=e^{-\rho (\tau_{x_1}^1-t)+c_2 \hat{L}_{\tau_{x_1}^1}^{2,t}}f_1(\tau_{x_1}^1,\hat{X}_{\tau_{x_1}^1}^{2,t,x_2}){\bf 1}_{\tau_{x_1}^1<T},
\end{align*}
which yields that
\begin{align*}
&\lim _{x_1'\downarrow x_1}\Ex\left[\int_{\tau_{x_1}^1\wedge T}^{\tau_{x_1'}^1\wedge T}e^{-\rho (\tau_{x_1}^1-t)+c_1 \hat{L}_{\tau_{x_1}^1}^{1',t}+c_2 \hat{L}_{\tau_{x_1}^1}^{2,t}}f_1(\tau_{x_1}^1,\hat{X}_{\tau_{x_1}^1}^{2,t,x_2})d\left(\frac{\hat{L}_s^{1,t}}{x_1'-x_1}\right)\right]\\
&\qquad=\Ex\left[e^{-\rho (\tau_{x_1}^1-t)+c_2 \hat{L}_{\tau_{x_1}^1}^{2,t}}f_1(\tau_{x_1}^1,\hat{X}_{\tau_{x_1}^1}^{2,t,x_2}){\bf 1}_{\tau_{x_1}^1<T}\right].
\end{align*}
The results above yield that
\begin{align}\label{eq:m-h-1}
&\lim _{x_1' \downarrow x_1}\mathbb{E}\left[\int_t^{T}e^{-\rho (s-t)+c_1 \hat{L}_s^{1',t}+c_2 \hat{L}_s^{2,t}}  f_1(\tau_{x_1}^1,\hat{X}_s^{2,t,x_2}) d\left(\frac{\hat{L}_s^{1',t}-\hat{L}_s^{1,t}}{x_1'-x_1}\right)\right]\\
&\qquad=-\Ex\left[e^{-\rho (\tau_{x_1}^1-t)+c_2 \hat{L}_{\tau_{x_1}^1}^{2,t}}f_1(s,\hat{X}_{\tau_{x_1}^1}^{2,t,x_2}){\bf 1}_{\tau_{x_1}^1<T}\right].\nonumber
\end{align}
Using the fact $\mathbb{P}(\tau^1_{x_1} = T) = 0$ (c.f. \citealt{Harrison85}) and following the same calculations of the representation \eqref{eq:varphi-x2-2}, we have
\begin{align}
&\lim _{x_1' \downarrow x_1}\mathbb{E}\left[\int_t^{T}  e^{-\rho (s-t)+c_2 \hat{L}_s^{2,t}}f_1(s,\hat{X}_s^{2,t,x_2})\left(\frac{e^{c_1 \hat{L}_s^{1',t}}-e^{c_1\hat{L}_s^{1,t}}}{x_1'-x_1}\right) d\hat{L}_s^{1,t}\right]\\
&\qquad=-c_1 \mathbb{E}\left[\int_{\tau_{x_1}^1\wedge T}^{T}  e^{-\rho (s-t)+c_1 \hat{L}_s^{1,t}+c_2 \hat{L}_s^{2,t}}f_1(s,\hat{X}_s^{2,t,x_2}) d\hat{L}_s^{1,t}\right].\nonumber
\end{align}
By the similar argument as above, we also have
\begin{align*}
\lim_{x_1' \uparrow x_1}\frac{ \varphi_1(t,x_1',x_2)- \varphi_1(t,x_1,x_2)}{x_1'-x_1}=\lim _{x_1' \downarrow x_1}\frac{ \varphi_1(t,x_1',x_2)- \varphi_1(t,x_1,x_2)}{x_1'-x_1}.
\end{align*}
Thus, we can conclude that, for all $(t,x_1,x_2)\in[0,T]\times\R^2_+$,
\begin{align}\label{eq:m-h}
\partial_{x_1} \varphi_1(t,x_1,x_2)&=\Ex\left[e^{-\rho (\tau_{x_1}^1-t)+c_2 \hat{L}_{\tau_{x_1}^1}^{2,t}}f_1(\tau_{x_1}^1,\hat{X}_{\tau_{x_1}^1}^{2,t,x_2}){\bf 1}_{\tau_{x_1}^1<T}\right]\nonumber\\
&\quad+c_1 \mathbb{E}\left[\int_{\tau_{x_1}^1\wedge T}^{T}  e^{-\rho (s-t)+c_1 \hat{L}_s^{1,t}+c_2 \hat{L}_s^{2,t}}f_1(s,\hat{X}_s^{2,t,x_2}) d\hat{L}_s^{1,t}\right].
\end{align}
Moreover, in view of \eqref{eq:RSDEi}, it follows that, for all $s\in[t,T]$,
\begin{align*}
\hat{L}_s^{1,t} &=x_1 \vee\left\{\max_{\ell\in[t,s]}\left(-\mu_1 (\ell-t) -\sum_{k=1}^m \sigma_{1k}(W_{\ell}^{k,1}-W_t^{k,1})\right)\right\}-x_1,\\
&=\left(\sup_{\ell\in[t,s]}\tilde{W}_{\ell}^{1,t}-x_1\right)^+.
\end{align*}
For $0\leq t\leq s\leq T$ and $x_1\in\R_+$, let us introduce
\begin{align}\label{eq:function-h}
h_1(t,s,x_1):=
\begin{cases}
\displaystyle \frac{1}{c_1}\partial_s\left(\Ex[e^{c_1 \hat{L}_s^{1,t}}]\right)=\frac{1}{c_1}\partial_s\left(\Ex\left[e^{c_1 \left(\sup_{\ell\in[t,s]}\tilde{W}_{\ell}^{1,t}-x_1\right)^+}\right]\right)\\[1.2em]
\displaystyle\qquad\qquad\qquad\qquad=\frac{1}{c_1}\partial_s\left(\int_0^{\infty} e^{c_1 (r -x_1)^+}p_1(t,s,r)dr\right), &c_1\neq 0;\\[1.2em]
\displaystyle\partial_s\left( \Ex[\hat{L}_s^{1,t}]\right)=\partial_s\left(\Ex\left[\left(\sup_{\ell\in[t,s]}\tilde{W}_{\ell}^{1,t}-x_1\right)^+\right]\right)\\[1.6em]
\displaystyle\qquad\qquad\qquad=\partial_s\left(\int_0^{\infty}(r-x_1)^+p_1(t,s,r)dr\right),&c_1=0.
\end{cases}
\end{align}
Here, $p_1(t,s,r)=\P(\sup_{\ell\in[t,s]}\tilde{W}_{\ell}^{1,t}\in dr)/dr$ is the density function of the maximum of drifted-Brownian motion, i.e., $\sup_{\ell\in[t,s]}\tilde{W}_{\ell}^{1,t}$ at time $s\in [t,T]$, which is given by
\begin{align*}
p_1(t,s,r)=\sqrt{\frac{2}{\pi \sigma_1^2(s-t)}} e^{-\frac{(r-\mu_1(s-t))^2}{2\sigma_1^2 (s-t)}}-\frac{2\mu_1}{\sqrt{2 \pi}\sigma_1^2} e^{\frac{2\mu_1}{\sigma_1^2}r} \int_{-\infty}^{\frac{-r-\mu_1 (s-t)}{\sqrt{\sigma_1^2 t}}}  \exp \left(-\frac{z^2}{2}\right) dz.
\end{align*}
 Hence, we deduce that, for all $(t,x_1,x_2)\in[0,T]\times\R^2_+$,
\begin{align}\label{eq:varphi-1-density}
\varphi_1(t,x_1,x_2)=
\int_t^{T}e^{-\rho (s-t)}\Ex\left[e^{c_2 \hat{L}_s^{2,t}}f_1(s,\hat{X}_{s}^{2,t,x_2})\right]h_1(t,s,x_1)ds.
\end{align}
Note that, for any fixed $s\in[t,T]$, the function $x_1\to  h_1(t,s,x_1)$ belongs to $C^2(\R_+)$, we get, for all $(t,x_1,x_2)\in[0,T]\times\R^2_+$,
\begin{align}\label{eq:m-hh}
\partial_{x_1x_1} \varphi_1(t,x_1,x_2)=
\int_t^{T}e^{-\rho (s-t)}\Ex\left[e^{c_2 \hat{L}_s^{2,t}}f_1(s,\hat{X}_{s}^{2,t,x_2})\right]\partial_{x_1x_1} h_1(t,s,x_1)ds.
\end{align}
By using \eqref{eq:varphi-1-density} and a similar argument as in the derivation of \eqref{varphi-x}, we have
\begin{align}\label{eq:m-rh}
&\partial_{x_1x_2}\varphi_1(t,x_1,x_2)=c_2\int_t^{T}e^{-\rho (s-t)}\Ex\left[e^{c_2 \hat{L}_s^{2,t}}f_1(s,\hat{X}_{s}^{2,t,x_2}){\bf 1}_{s\geq \tau_{x_2}^2}\right]\partial_{x_1} h_1(t,s,x_1)ds\\
&\qquad\qquad\qquad\qquad-\int_t^{T}e^{-\rho (s-t)}\Ex\left[\partial_{x_2}f_1(s,\hat{X}_{s}^{2,t,x_2}){\bf 1}_{s\leq \tau_{x_2}^2}\right]\partial_{x_1} h_1(t,s,x_1)ds.\nonumber
\end{align}

Finally, we examine the expression of the partial derivative $\partial_t \varphi_1(t,x_1,x_2)$. Consider the case $c_1\neq 0$, for any $\delta\in[0,T-t]$, it holds that
\begin{align*}
\varphi_1(t+\delta,x_1,x_2)
&=-\frac{1}{c_1}\int_{t+\delta}^{T} e^{-\rho(s-t-\delta)}\Ex\left[e^{c_2 \hat{L}_s^{2,t+\delta}}f_1(s,\hat{X}_s^{2,t+\delta,x_2})\right]d\Ex[e^{c_1\hat{L}_s^{1,t+\delta}}]\nonumber\\
&=-\frac{1}{c_1}\int_{t}^{T-\delta} e^{-\rho(s-t)}\Ex\left[e^{c_2 \hat{L}_{s+\delta}^{2,t+\delta}}f_1(s+\delta,\hat{X}_{s+\delta}^{2,t+\delta,x_2})\right]d\Ex[e^{c_1\hat{L}_{s+\delta}^{1,t+\delta}}]\\
&=-\frac{1}{c_1}\int_{t}^{T-\delta} e^{-\rho(s-t)}\Ex\left[e^{c_2 \hat{L}_s^{2,t}}f_1(s+\delta,\hat{X}_s^{2,t,x_2})\right]d\Ex[e^{c_1\hat{L}_s^{1,t}}].
\end{align*}
Then, we arrive at
\begin{align*}
&\frac{\varphi_1(t+\delta,x_1,x_2)-\varphi_1(t,x_1,x_2)}{\delta}\\
&\quad=-\frac{1}{c_1\delta}\Bigg\{\int_{t}^{T-\delta} e^{-\rho(s-t)}\Ex\left[e^{c_2 \hat{L}_s^{2,t}}f_1(s+\delta,\hat{X}_s^{2,t,x_2})\right]d\Ex[e^{c_1\hat{L}_s^{1,t}}]\nonumber\\
&\quad\qquad-\int_{t}^{T} e^{-\rho(s-t)}\Ex\left[e^{c_2 \hat{L}_s^{2,t}}f_1(s,\hat{X}_s^{2,t,x_2})\right]d\Ex[e^{c_1\hat{L}_s^{1,t}}] \Bigg\}\\
&\quad=-\frac{1}{c_1}\int_{t}^{T} e^{-\rho(s-t)}\frac{ 1}{\delta}\left[e^{c_2 \hat{L}_s^{2,t}}\left(f_1(s+\delta,\hat{X}_s^{2,t,x_2})-f_1(s,\hat{X}_s^{2,t,x_2})\right)\right]d\Ex[e^{c_1\hat{L}_s^{1,t}}]\\
&\quad\quad+\frac{1}{c_1\delta}\int_{T-\delta}^T e^{-\rho(s-t)}\Ex\left[e^{c_2 \hat{L}_s^{2,t}}f_1(s,\hat{X}_s^{2,t,x_2})\right]d\Ex[e^{c_1\hat{L}_s^{1,t}}]\\
&\quad=:\Delta_{\delta}^{(1)}+\Delta_{\delta}^{(2)}.
\end{align*}
Note that $\lim_{\delta\to 0}\frac{f_1(s+\delta,\hat{X}_{s}^{2,t,x_2})-f_1(s,\hat{X}_{s}^{2,t,x_2})}{\delta}=\partial_t f_1(s,\hat{X}_{s}^{2,t,x_2})$.
Thus, it follows from Assumption {\bf(A)} and DCT that
\begin{align}\label{eq:deltaD1}
\lim_{\delta\to 0}\Delta_{\delta}^{(1)}&=-\lim_{\delta\to 0}\frac{1}{c_1}\int_{t}^{T} e^{-\rho(s-t)}\frac{ 1}{\delta}\left[e^{c_2 \hat{L}_s^{2,t}}\left(f_1(s+\delta,\hat{X}_s^{2,t,x_2})-f_1(s,\hat{X}_s^{2,t,x_2})\right)\right]d\Ex[e^{c_1\hat{L}_s^{1,t}}]\nonumber\\
&=-\frac{1}{c_1}\int_{t}^{T} e^{-\rho(s-t)}\Ex\left[e^{c_2 \hat{L}_s^{2,t}}\partial_t f_1(s,\hat{X}_{s}^{2,t,x_2})\right]d\Ex[e^{c_1\hat{L}_s^{1,t}}].
\end{align}
We next deal with the term $\Delta_{\delta}^{(2)}$. In fact, we have
\begin{align*}
\Delta_{\delta}^{(2)}
&=\frac{1}{c_1\delta}\int_{T-\delta}^T \bigg\{e^{-\rho(s-t)}\Ex\left[e^{c_2 \hat{L}_s^{2,t}}f_1(s,\hat{X}_s^{2,t,x_2})\right]-e^{-\rho(T-t)}\Ex\left[e^{c_2 \hat{L}_T^{2,t}}f_1(T,\hat{X}_T^{2,t,x_2})\right]\bigg\}d\Ex[e^{c_1\hat{L}_s^{1,t}}]\\
&\quad+\frac{1}{c_1\delta}\int_{T-\delta}^T e^{-\rho(T-t)}\Ex\left[e^{c_2 \hat{L}_T^{2,t}}f_1(T,\hat{X}_{T}^{2,t,x_2})\right]d\Ex[e^{c_1\hat{L}_s^{1,t}}]\\
&:=I_{\delta}+J_{\delta}.
\end{align*}
First of all, it holds from \eqref{eq:function-h} that
\begin{align*}
|I_{\delta}|&=\frac{1}{c_1\delta}\int_{T-\delta}^T \bigg|e^{-\rho(s-t)}\Ex\left[e^{c_2 \hat{L}_s^{2,t}}f_1(s+\delta,\hat{X}_{s}^{2,t,x_2})\right]\\
&\qquad\qquad\qquad\quad-e^{-\rho(T-t)}\Ex\left[e^{c_2 \hat{L}_T^{2,t}}f_1(T,\hat{X}_{T}^{2,t,x_2})\right]\bigg|d\Ex[e^{c_1\hat{L}_s^{1,t}}]\\
&\leq \frac{1}{c_1\delta}\int_{T-\delta}^{T}\Ex\bigg[ \max_{s\in[T-\delta,T]}\bigg|e^{-\rho(s-t)+c_2 \hat{L}_s^{2,t}}f_1(s+\delta,\hat{X}_{s}^{2,t,x_2})\\
&\qquad\qquad\qquad\quad-e^{-\rho(T-t)+c_2 \hat{L}_T^{2,t}}f_1(T,\hat{X}_{T}^{2,t,x_2})\bigg|\bigg]d\Ex[e^{c_1\hat{L}_s^{1,t}}]\\
&=\frac{1}{c_1}\Ex\left[\max_{s\in[T-\delta,T]}\left|e^{-\rho(s-t)+c_2 \hat{L}_s^{2,t}}f_1(s+\delta,\hat{X}_{s}^{2,t,x_2})-e^{-\rho(T-t)+c_2 \hat{L}_T^{2,t}}f_1(T,\hat{X}_{T}^{2,t,x_2})\right|\right]\\
&\qquad\times \frac{1}{\delta}\int_{T-\delta}^{T}d\Ex[e^{c_1\hat{L}_s^{1,t}}]\\
&=\frac{1}{c_1}\Ex\left[\max_{s\in[T-\delta,T]}\left|e^{-\rho(s-t)+c_2 \hat{L}_s^{2,t}}f_1(s+\delta,\hat{X}_{s}^{2,t,x_2})-e^{-\rho(T-t)+c_2 \hat{L}_T^{2,t}}f_1(T,\hat{X}_{T}^{2,t,x_2})\right|\right]\\
&\qquad\times\frac{1}{\delta}\int_{T-\delta}^{T} h_1(t,s,x_1)ds.
\end{align*}
We obtain from DCT that $\lim_{\delta\to 0}I_{\delta}=0$. For the term $J_{\delta}$, it follows from \eqref{eq:function-h} that
\begin{align*}
\lim_{\delta \to 0}J_{\delta}&=\lim_{\delta \to 0}\frac{1}{c_1 \delta}\int_{T-\delta}^T  e^{-\rho(T-t)}\Ex\left[e^{c_2 \hat{L}_T^{2,t}}f_1(T,\hat{X}_{T}^{2,t,x_2})\right]d\Ex[e^{c_1\hat{L}_s^{1,t}}]\\
&=\frac{1}{c_1}e^{-\rho(T-t)}\Ex\left[e^{c_2 \hat{L}_T^{2,t}}f_1(T,\hat{X}_{T}^{2,t,x_2})\right]\lim_{\delta \to 0}\left\{\frac{1}{\delta}\int_{T-\delta}^Td\Ex\left[e^{c_1\hat{L}_s^{1,t}}\right]\right\}\\
&=\frac{1}{c_1}e^{-\rho(T-t)}\Ex\left[e^{c_2 \hat{L}_T^{2,t}}f_1(T,\hat{X}_{T}^{2,t,x_2})\right] h_1(t,T,x_1).
\end{align*}
The calculation of the case $c_1=0$ can be obtained by a similar argument. Hence, we deduce that
\begin{align}\label{eq:m-t}
\partial_t \varphi_1(t,x_1,x_2)=
\begin{cases}
\displaystyle -\frac{1}{c_1}\int_{t}^{T} e^{-\rho(s-t)}\Ex\left[e^{c_2 \hat{L}_s^{2,t}}\partial_t f_1(s,\hat{X}_{s}^{2,t,x_2})\right]d\Ex[e^{c_1\hat{L}_s^{1,t}}]\\[1.2em]
\displaystyle\qquad\quad+\frac{1}{c_1}e^{-\rho(T-t)}\Ex\left[e^{c_2 \hat{L}_T^{2,t}}f_1(T,\hat{X}_{T}^{2,t,x_2})\right] h_1(t,T,x_1), &c_1\neq 0;\\[1.2em]
\displaystyle-\int_{t}^{T} e^{-\rho(s-t)}\Ex\left[e^{c_2 \hat{L}_s^{2,t}}\partial_t f_1(s,\hat{X}_{s}^{2,t,x_2})\right]d\Ex[\hat{L}_s^{1,t}]\\[1.2em]
\displaystyle\qquad\quad+e^{-\rho(T-t)}\Ex\left[e^{c_2 \hat{L}_T^{2,t}}f_1(T,\hat{X}_{T}^{2,t,x_2})\right] h_1(t,T,x_1), &c_1=0.
\end{cases}
\end{align}
We then conclude that $\varphi_1\in C^{1,2}([0,T)\times\R_+^2)\cap C([0,T]\times\overline{\R}_+^2)$ by combining results in \eqref{varphi-x}, \eqref{eq:m-rr1}, \eqref{eq:m-h}, \eqref{eq:m-hh}, \eqref{eq:m-rh} and \eqref{eq:m-t}. Thus, we complete the proof of the lemma.
\end{proof}

\begin{proof}[Proof of Proposition~\ref{prop:varphisol}] We first prove item (i). Based on Lemma \ref{lem:regularity-varphi}, we will show that $\varphi(t,\bf{x})$ satisfies the Robin boundary problem \eqref{eq:PDE-multidimL0}. Consider using \eqref{varphi-x} and \eqref{eq:m-h}, we can verify that $\varphi$ satisfies the following Robin boundary condition, for all $ (t,\mathbf{x})\in[0,T)\times \R_+^d$,
\begin{align}\label{eq:Neu}
 \partial_{x_i}\varphi(t,({\bf x}^{-i},0))+c_i \varphi(t,({\bf x}^{-i},0))=f_i(t,{\bf x}^{-i}),\quad  i=1,\ldots,d.
\end{align}
For $i=1,\ldots,d$, we recall the scalar Brownian motion $\tilde{W}^{i,t}=(\tilde{W}^{i,t}_s)_{s\in[t,T]}$ which is given in \eqref{eq:tildeW}. For any $\epsilon\in(0,\min(x_1,\ldots,x_d))$, we define
\begin{align}
\tau_{\epsilon}^t:=\inf\{s\geq t;~|\tilde{W}_s^{i,t}|\geq \epsilon,~\forall i=1,...,d\}
\end{align}
with $\inf\emptyset=+\infty$ by convention. Then, in lieu of \eqref{eq:RSDEi}, one can easily check that $\hat{X}^{i,t,x_i}_{\hat{t}\land \tau_{\epsilon}^t}=x_i+\tilde{W}^{i,t}_{\hat{t}\land \tau_{\epsilon}^t}$ for any $\hat{t}\in[t,T]$. Thus, it follows from the strong Markov property that
\begin{align*}
&-\sum_{i=1}^d\Ex\left[\int_{\hat{t}\wedge\tau_{\epsilon}^t}^{T} e^{-\rho s+\sum_{k=1}^d c_k \hat{L}_s^{k,t}}f_i(s,\hat{\mathbf{X}}_s^{-i,t,\mathbf{x}})d\hat{L}_s^{i,t}|\mathcal{ F}_{\hat{t}\wedge\tau_{\epsilon}^t}\right]=\tilde{\varphi}\left(\hat{t}\wedge\tau_{\epsilon}^t,\hat{\mathbf{X}}^{t,\mathbf{x}}_{\hat{t}\wedge\tau_{\epsilon}^t}\right),
\end{align*}
where we have used the fact from \eqref{eq:varphisol} that
\begin{align}\label{eq:tilde-rho}
\tilde{\varphi}(t,{\bf x})&:=-\sum_{i=1}^d \Ex\left[ \int_t^{T} e^{-\rho s+\sum_{k=1}^d c_k \hat{L}_s^{k,t}}f_i(s,\hat{\mathbf{X}}_s^{-i,t,\mathbf{x}})d\hat{L}_s^{i,t}\right]=e^{-\rho t}\varphi(t,{\bf x}).
\end{align}
Hence, for all $(t,{\bf x})\in [0,T]\times\R_+^d$, it holds that
\begin{align*}
\tilde{\varphi}(t,{\bf x})&=\Ex\left[\tilde{\varphi}\left(\hat{t}\wedge\tau_{\epsilon}^t,\hat{\mathbf{X}}^{t,\mathbf{x}}_{\hat{t}\wedge\tau_{\epsilon}^t}\right)-\sum_{i=1}^d\int_t^{\hat{t}\wedge\tau_{\epsilon}^t} e^{-\rho s+\sum_{k=1}^d c_k \hat{L}_s^{k,t}}f_i(s,\hat{\mathbf{X}}_s^{-i,t,\mathbf{x}})d\hat{L}_s^{i,t}\right].
\end{align*}
It follows from It{\^o}'s formula that
\begin{align}\label{eq:hat-t}
&\frac{1}{\hat{t}-t}\Ex\left[\sum_{i=1}^d\int_t^{\hat{t}\wedge\tau_{\epsilon}^t} e^{-\rho s+\sum_{k=1}^d c_k \hat{L}_s^{k,t}}f_i(s,\hat{\mathbf{X}}_s^{-i,t,\mathbf{x}})d\hat{L}_s^{i,t}\right]\nonumber\\
&\quad
=\frac{1}{\hat{t}-t}\Ex\left[\tilde{\varphi}\left(\hat{t}\wedge\tau_{\epsilon}^t,\hat{\mathbf{X}}^{t,\mathbf{x}}_{\hat{t}\wedge\tau_{\epsilon}^t}\right)
-\tilde{\varphi}(t,\bf{x})\right]=\frac{1}{\hat{t}-t}\Ex\left[\int_t^{\hat{t}\wedge\tau_{\epsilon}^t}(\partial_t+\mathcal{L}^0)\tilde{\varphi}\left(s,\hat{\mathbf{X}}^{t,\mathbf{x}}_{s}\right)d s\right]\nonumber\\
&\qquad\quad+\frac{1}{\hat{t}-t} \Ex\left[\sum_{i=1}^d\int_t^{\hat{t}\wedge\tau_{\epsilon}^t} \left(\partial_{x_i} \tilde{\varphi}+c_i \tilde{\varphi}\right)\left(s,\hat{\mathbf{X}}_s^{t,\mathbf{x}}\right)d\hat{L}_s^{i,t}\right]\nonumber\\
&\quad=\frac{1}{\hat{t}-t}\Ex\left[\int_t^{\hat{t}\wedge\tau_{\epsilon}^t}  e^{-\rho s}(\partial_t +\mathcal{L}^0-\rho)\varphi\left(s,\hat{\mathbf{X}}^{t,\mathbf{x}}_{s}\right)d s\right]\nonumber\\
&\quad\quad+\frac{1}{\hat{t}-t} \Ex\left[\sum_{i=1}^d\int_t^{\hat{t}\wedge\tau_{\epsilon}^t}  e^{-\rho s}(\partial_{x_i} \varphi+c_i \varphi)\left(s,\hat{\mathbf{X}}_s^{t,\mathbf{x}}\right)d\hat{L}_s^{i,t}\right].
\end{align}
Then, the DCT yields that
\begin{align}\label{eq:hat-t-2}
&\lim_{\hat{t}\downarrow t}\frac{1}{\hat{t}-t}\Ex\left[\int_t^{\hat{t}\wedge\tau_{\epsilon}^t}  e^{-\rho s+\sum_{k=1}^d c_k \hat{L}_s^{k,t}}\left(\partial_t+\mathcal{L}^0-\rho\right) \varphi\left(s,\hat{\mathbf{X}}^{t,\mathbf{x}}_{s}\right)d s\right]\nonumber\\
&\qquad=e^{-\rho t}\left(\partial_t+\mathcal{L}^0-\rho \right)\varphi(t,{\bf x}).
\end{align}
Note that $\hat{L}_s^{i,t}=0$ on $s\in[t,\hat{t}\wedge\tau_{\epsilon}^t]$.  By using \eqref{eq:hat-t} and \eqref{eq:hat-t-2} , we obtain that $(\partial_t \varphi+\mathcal{L}^0\varphi-\rho \varphi)(t,{\bf x})=0$ on $(t,{\bf x})\in [0,T)\times\R_+^d$.  Therefore, we deduce that the function $\varphi$ defined by the probabilistic representation \eqref{eq:varphisol} satisfies the Robin boundary problem \eqref{eq:PDE-multidimL0}.

We next prove item (ii). Assume that the Robin boundary problem \eqref{eq:PDE-multidimL0} has a classical solution $\varphi$ satisfying that there exists a constant $C>0$ and $q\geq 1$ such that $|\varphi(t,\mathbf{x})|\leq C(1+|\mathbf{x}|^q)$. For $n\geq 1$ and $t\in[0,T]$, we introduce the stopping time defined by
\begin{align*}
\tau_{t,T}^n:=\inf\{s\geq t;~\min\{X_s^{1,t,x_1},\ldots,X_s^{d,t,x_2}\}\geq n\}\wedge T.
\end{align*}
It follows from It{\^o}'s formula that, for all $(t,{\bf x})\in[0,T]\times\R_+^{d}$,
\begin{align}\label{eq:HJB-m-cla-sol}
&\Ex\left[ e^{\sum_{k=1}^d c_k \hat{L}_{\tau_{t,T}^n}^{k,t}}\tilde{\varphi}\left(\tau_{t,T}^n,\mathbf{\hat{X}}^{t,\mathbf{x}}_{\tau_{t,T}^n}\right)\right]\\
&\qquad = \tilde{\varphi}(t,{\bf x})+\sum_{i=1}^d \Ex\left[\int_t^{\tau_{t,T}^n} e^{\sum_{k=1}^d c_k \hat{L}_s^{k,t}}\left(\partial_{x_i} \tilde{\varphi}+c_i \tilde{\varphi}\right)(s,\hat{\mathbf{X}}_s^{t,\mathbf{x}})  d \hat{L}_s^{i,t}\right] \nonumber\\
&\qquad\quad+\Ex\left[\int_t^{\tau_{t,T}^n} e^{\sum_{k=1}^d c_k \hat{L}_s^{k,t}}\left(\partial_t \tilde{\varphi}+\mathcal{L}^0\tilde{\varphi}\right)(s,\mathbf{\hat{X}}^{t,\mathbf{x}}_s) d s\right]\nonumber \\
&\qquad=e^{-\rho t}\varphi(t,{\bf x})+\sum_{i=1}^d \Ex\left[\int_t^{\tau_{t,T}^n} e^{-\rho s+\sum_{k=1}^d c_k \hat{L}_s^{k,t}}\left(\partial_{x_i} \varphi+c_i \varphi\right)(s,\hat{\mathbf{X}}_s^{t,\mathbf{x}})  d \hat{L}_s^{i,t}\right] \nonumber\\
&\qquad\quad+\Ex\left[\int_t^{\tau_{t,T}^n} e^{-\rho s+\sum_{k=1}^d c_k \hat{L}_s^{k,t}}\left(\partial_t \varphi+\mathcal{L}^0\varphi-\rho \varphi\right)(s,\mathbf{\hat{X}}^{t,\mathbf{x}}_s) d s\right]\nonumber \\
&\qquad=e^{-\rho t} \varphi(t,{\bf x})+\sum_{i=1}^d\Ex\left[\int_t^{\tau_{t,T}^n}  e^{-\rho s+\sum_{k=1}^d c_k \hat{L}_s^{k,t}}f_i(s,\hat{\mathbf{X}}_s^{-i,t,\mathbf{x}})d\hat{L}_s^{i,t}\right].\nonumber
\end{align}
 Moreover, by using the relationship $\tilde{\varphi}(t,{\bf x})=e^{-\rho t}\varphi(t,{\bf x})$ from \eqref{eq:tilde-rho}, the DCT and the polynomial growth condition of $\varphi$ on $[0,T]\times \R_+^d$, we have
\begin{align*}
\lim_{n\to \infty}\Ex\left[e^{\sum_{k=1}^d c_k \hat{L}_{\tau_{t,T}^n}^{k,t}}\tilde{\varphi}\left(\tau_{t,T}^n,\hat{\mathbf{X}}^{t,\mathbf{x}}_{\tau_{t,T}^n}\right)\right]&=\lim_{n\to \infty}\Ex\left[e^{-\rho \tau_{t,T}^n+\sum_{k=1}^d c_k \hat{L}_{\tau_{t,T}^n}^{k,t}}\varphi\left(\tau_{t,T}^n,\hat{\mathbf{X}}^{t,\mathbf{x}}_{\tau_{t,T}^n}\right)\right]\\
&=\Ex\left[e^{-\rho T+\sum_{k=1}^d c_k \hat{L}_{T}^{k,t}}\varphi\left(T,\hat{\mathbf{X}}^{t,\mathbf{x}}_{T}\right)\right]=0.
\end{align*}
Letting $n \rightarrow \infty$ on both sides of the equality \eqref{eq:HJB-m-cla-sol} and using DCT and MCT, we obtain the probabilistic representation \eqref{eq:varphisol} of the solution $\varphi(t,{\bf x})$, which ends the proof of the proposition.
\end{proof}

\begin{proof}[Proof of Lemma \ref{lem:varphixixj}]
By applying \eqref{eq:m-rh} in the proof of Lemma \ref{lem:regularity-varphi}, we have
\begin{align*}
\begin{cases}
\displaystyle \partial_{x_1x_j}\varphi_1(t,{\bf x})=c_j\int_t^{T}e^{-\rho (s-t)}\Ex\left[e^{\sum_{k=2}^d c_k \hat{L}_s^{k,t}}f_1(s,\hat{{\bf X}}_{s}^{-1,t,{\bf x}}){\bf 1}_{s\geq \tau_{x_j}^j}\right]\partial_{x_1} h_1(t,s,x_1)ds\\[1em]
\displaystyle \quad\qquad-\int_t^{T}e^{-\rho (s-t)}\Ex\left[e^{\sum_{k=2}^ d c_k \hat{L}_s^{k,t}}\partial_{x_j}f_1(s,\hat{{\bf X}}_{s}^{-1,t,{\bf x}}){\bf 1}_{s\leq \tau_{x_j}^j}\right]\partial_{x_1} h_1(t,s,x_1)ds,\\[1em]
\displaystyle \hfill 1<j\leq d;\\
\displaystyle\partial_{x_ix_j}\varphi_1(t,{\bf x})=-c_ic_j\int_t^{T}e^{-\rho (s-t)}\Ex\left[e^{\sum_{k=2}^d c_k \hat{L}_s^{k,t}}f_1(s,\hat{{\bf X}}_{s}^{-1,t,{\bf x}}){\bf 1}_{s\geq \tau_{x_i}^i\vee\tau_{x_j}^j}\right] h_1(t,s,x_1)ds\\[1em]
\displaystyle\quad\qquad+c_j\int_t^{T}e^{-\rho (s-t)}\Ex\left[e^{\sum_{k=2}^d c_k \hat{L}_s^{k,t}}\partial_{x_i} f_1(s,\hat{{\bf X}}_{s}^{-1,t,{\bf x}}){\bf 1}_{\tau_{x_i}^i\geq s\geq \tau_{x_j}^j}\right] h_1(t,s,x_1)ds\\[1em]
\displaystyle\quad\qquad+c_i\int_t^{T}e^{-\rho (s-t)}\Ex\left[e^{\sum_{k=2}^ d c_k \hat{L}_s^{k,t}}\partial_{x_j}f_1(s,\hat{{\bf X}}_{s}^{-1,t,{\bf x}}){\bf 1}_{\tau_{x_j}^j\geq s\geq \tau_{x_i}^i}\right] h_1(t,s,x_1)ds\\[1em]
\displaystyle\quad\qquad-\int_t^{T}e^{-\rho (s-t)}\Ex\left[e^{\sum_{k=2}^ d c_k \hat{L}_s^{k,t}}\partial_{x_i x_j}f_1(s,\hat{{\bf X}}_{s}^{-1,t,{\bf x}}){\bf 1}_{s\leq \tau_{x_i}^i\wedge\tau_{x_j}^j}\right] h_1(t,s,x_1)ds,\\
\displaystyle \hfill 1<i<j\leq d.
\end{cases}
\end{align*}
In view of \eqref{eq:RSDEi}, it holds that, for $i=1,\ldots,d$,
\begin{align}\label{eq:L2-max}
\hat{L}_s^{i,t} &=x_i \vee\left\{\max_{\ell\in[t,s]}\left(-\mu_s (\ell-t) -\sum_{k=1}^m \sigma_{sk}(W_{\ell}^{k,i}-W_t^{k,i})\right)\right\}-x_i,\nonumber\\
&=\left(\sup_{\ell\in[t,s]}\tilde{W}_{\ell}^{i,t}-x_i\right)^+.
\end{align}
Then, it follows from \eqref{eq:Xi-max} and \eqref{eq:L2-max} that
\begin{align*}
&\int_t^{T}e^{-\rho (s-t)}\Ex\left[e^{\sum_{k=2}^d c_k \hat{L}_s^{k,t}}f_1(s,\hat{{\bf X}}_{s}^{-1,t,{\bf x}}){\bf 1}_{s\geq \tau_{x_j}^j}\right]\partial_{x_1} h_1(t,s,x_1)ds\\
&=\int_t^{T}e^{-\rho (s-t)}\Ex\left[e^{\sum_{k=2}^d c_k \left(\sup_{\ell\in[t,s]}\tilde{W}_{\ell}^{k,t}-x_k\right)^+}f_1\left(s,{\bf x}^{-1}-\tilde{{\bf W}}_s^{-1,t}+\left(\sup_{\ell\in[t,s]}\tilde{{\bf W}}_{\ell}^{-1,t}-{\bf x}^{-1}\right)^+\right)\right.\\
&\quad\qquad\qquad\left.\times{\bf 1}_{\sup_{\ell\in[t,s]}\tilde{W}_{\ell}^{j,t}\geq x_j}\right]\partial_{x_1} h_1(t,s,x_1)ds\\
&=\int_t^{T}\int_{0}^{\infty}\int_{-\infty}^{y_2}\cdots\int_{x_j}^{\infty}\int_{-\infty}^{y_j}\cdots\int_{0}^{\infty}\int_{-\infty}^{y_d}e^{-\rho(s-t)+\sum_{k=2}^d c_k (y_k-x_k)}\\
&\qquad\times f_1(s,{\bf x}^{-1}-{\bf r}^{-1}+({\bf y}^{-1}-{\bf x}^{-1})^+)\partial_{x_1} h_1(t,s,x_1) \Pi_{k=2}^d\phi_k(t,s,r_k,y_k)dr_2dy_2...dr_d dy_dds.
\end{align*}
In a similar fashion, we can also obtain that
{\small\begin{align*}
&\int_t^{T}e^{-\rho (s-t)}\Ex\bigg[e^{\sum_{k=2}^ d c_k \hat{L}_s^{k,t}}\partial_{x_j}f_1(s,\hat{{\bf X}}_{s}^{-j,t,{\bf x}}){\bf 1}_{s\leq \tau_{x_j}^j}\bigg]\partial_{x_1} h_1(t,s,x_1)ds\\
&=\int_t^{T}e^{-\rho (s-t)}\Ex\bigg[e^{\sum_{k=2}^d c_k \left(\sup_{\ell\in[t,s]}\tilde{W}_{\ell}^{k,t}-x_k\right)^+}\partial_{x_j} f_1\bigg(s,{\bf x}^{-1}-\tilde{{\bf W}}_s^{-1,t}+\bigg(\sup_{\ell\in[t,s]}\tilde{{\bf W}}_{\ell}^{-1,t}-{\bf x}^{-1}\bigg)^+\bigg)\\
&\quad\qquad\qquad\times{\bf 1}_{\sup_{\ell\in[t,s]}\tilde{W}_{\ell}^{j,t}\leq x_j}\bigg]\partial_{x_1} h_1(t,s,x_1)ds\\
&=\int_t^{T}\int_{0}^{\infty}\int_{-\infty}^{y_2}\cdots\int_0^{x_j}\int_{-\infty}^{y_j}\cdots\int_{0}^{\infty}\int_{-\infty}^{y_d}e^{-\rho(s-t)+\sum_{k=2}^d c_k (y_k-x_k)}\nonumber\\
&\qquad\times\partial_{x_j}f_1(s,{\bf x}^{-1}-{\bf r}^{-1}+({\bf y}^{-1}-{\bf x}^{-1})^+)\partial_{x_1} h_1(t,s,x_1) \Pi_{k=2}^d\phi_k(t,s,r_k,y_k)dr_2dy_2...dr_d dy_dds,\\[1em]
&\int_t^{T}e^{-\rho (s-t)}\Ex\bigg[e^{\sum_{k=2}^d c_k \hat{L}_s^{k,t}}f_1(s,\hat{{\bf X}}_{s}^{-1,t,{\bf x}}){\bf 1}_{s\geq \tau_{x_i}^i\vee\tau_{x_j}^j}\bigg] h_1(t,s,x_1)ds\\
&=\int_t^{T}e^{-\rho (s-t)}\Ex\bigg[e^{\sum_{k=2}^d c_k \left(\sup_{\ell\in[t,s]}\tilde{W}_{\ell}^{k,t}-x_k\right)^+} f_1\bigg(s,{\bf x}^{-1}-\tilde{{\bf W}}_s^{-1,t}+\bigg(\sup_{\ell\in[t,s]}\tilde{{\bf W}}_{\ell}^{-1,t}-{\bf x}^{-1}\bigg)^+\bigg)\\
&\quad\qquad\qquad\times{\bf 1}_{\sup_{\ell\in[t,s]}\tilde{W}_{\ell}^{i,t}\geq x_i,~\sup_{\ell\in[t,s]}\tilde{W}_{\ell}^{j,t}\geq x_j}\bigg]h_1(t,s,x_1)ds\\
&=\int_t^{T}\int_{0}^{\infty}\int_{-\infty}^{y_2}\cdots\int_{x_i}^{\infty}\int_{-\infty}^{y_i}\cdots\int_{x_j}^{\infty}\int_{-\infty}^{y_j}\cdots\int_{0}^{\infty}\int_{-\infty}^{y_d}e^{-\rho(s-t)+\sum_{k=2}^d c_k (y_k-x_k)}\\
&\quad\qquad\qquad\times f_1(s,{\bf x}^{-1}-{\bf r}^{-1}+({\bf y}^{-1}-{\bf x}^{-1})^+) h_1(t,s,x_1) \Pi_{k=2}^d\phi_k(t,s,r_k,y_k)dr_2dy_2...dr_d dy_dds.\\[1em]
&\int_t^{T}e^{-\rho (s-t)}\Ex\bigg[e^{\sum_{k=2}^d c_k \hat{L}_s^{k,t}}\partial_{x_i} f_1(s,\hat{{\bf X}}_{s}^{-1,t,{\bf x}}){\bf 1}_{\tau_{x_i}^i\geq s\geq \tau_{x_j}^j}\bigg] h_1(t,s,x_1)ds\\
&=\int_t^{T}e^{-\rho (s-t)}\Ex\bigg[e^{\sum_{k=2}^d c_k \left(\sup_{\ell\in[t,s]}\tilde{W}_{\ell}^{k,t}-x_k\right)^+}\partial_{x_i} f_1\bigg(s,{\bf x}^{-1}-\tilde{{\bf W}}_s^{-1,t}+\bigg(\sup_{\ell\in[t,s]}\tilde{{\bf W}}_{\ell}^{-1,t}-{\bf x}^{-1}\bigg)^+\bigg)\\
&\quad\qquad\qquad\times{\bf 1}_{\sup_{\ell\in[t,s]}\tilde{W}_{\ell}^{i,t}\leq x_i,~\sup_{\ell\in[t,s]}\tilde{W}_{\ell}^{j,t}\geq x_j}\bigg]\partial_{x_1} h_1(t,s,x_1)ds\\
&=\int_t^{T}\int_{0}^{\infty}\int_{-\infty}^{y_2}\cdots\int_0^{x_i}\int_{-\infty}^{y_i}\cdots\int_{x_j}^{\infty}\int_{-\infty}^{y_j}\cdots\int_{0}^{\infty}\int_{-\infty}^{y_d}e^{-\rho(s-t)+\sum_{k=2}^d c_k (y_k-x_k)}\\
&\quad\qquad\times \partial_{x_i}f_1(s,{\bf x}^{-1}-{\bf r}^{-1}+({\bf y}^{-1}-{\bf x}^{-1})^+) h_1(t,s,x_1) \Pi_{k=2}^d\phi_k(t,s,r_k,y_k)dr_2dy_2...dr_d dy_dds,\\[1em]
&\int_t^{T}e^{-\rho (s-t)}\Ex\left[e^{\sum_{k=2}^ d c_k \hat{L}_s^{k,t}}\partial_{x_j}f_1(s,\hat{{\bf X}}_{s}^{-1,t,{\bf x}}){\bf 1}_{\tau_{x_j}^j\geq s\geq \tau_{x_i}^i}\right] h_1(t,s,x_1)ds\\
&=\int_t^{T}\int_{0}^{\infty}\int_{-\infty}^{y_2}\cdots\int_{x_i}^{\infty}\int_{-\infty}^{y_i}\cdots\int_0^{x_j}\int_{-\infty}^{y_j}\cdots\int_{0}^{\infty}\int_{-\infty}^{y_d}e^{-\rho(s-t)+\sum_{k=2}^d c_k (y_k-x_k)}\\
&\quad\qquad\times \partial_{x_j}f_1(s,{\bf x}^{-1}-{\bf r}^{-1}+({\bf y}^{-1}-{\bf x}^{-1})^+)h_1(t,s,x_1) \Pi_{k=2}^d\phi_k(t,s,r_k,y_k)dr_2dy_2...dr_d dy_dds,\\
&\int_t^{T}e^{-\rho (s-t)}\Ex\left[e^{\sum_{k=2}^ d c_k \hat{L}_s^{k,t}}\partial_{x_i x_j}f_1(s,\hat{{\bf X}}_{s}^{-1,t,{\bf x}}){\bf 1}_{s\leq \tau_{x_i}^i\wedge\tau_{x_j}^j}\right] h_1(t,s,x_1)ds\\
&=\int_t^{T}\int_{0}^{\infty}\int_{-\infty}^{y_2}\cdots\int_0^{x_i}\int_{-\infty}^{y_i}\cdots\int_0^{x_j}\int_{-\infty}^{y_j}
\cdots\int_{0}^{\infty}\int_{-\infty}^{y_d}e^{-\rho(s-t)+\sum_{k=2}^d c_k (y_k-x_k)}\\
&\quad\qquad\times \partial_{x_ix_j}f_1(s,{\bf x}^{-1}-{\bf r}^{-1}+({\bf y}^{-1}-{\bf x}^{-1})^+) h_1(t,s,x_1) \Pi_{k=2}^d\phi_k(t,s,r_k,y_k)dr_2dy_2...dr_d dy_dds.
\end{align*}}
Thus, the desired result follows from the equality $\partial_{x_ix_j}\varphi(t,{\bf x})=\sum_{k=1}^d \partial_{x_ix_j}\varphi_k(t,{\bf x})$.
\end{proof}

\noindent
\textbf{Acknowledgements.}\quad {\small The authors are grateful to the anonymous referee for the helpful comments and suggestions.  L. Bo and Y. Huang are supported by National Natural Science of Foundation of China (No. 12471451), Natural Science Basic Research Program of Shaanxi (No. 2023-JC-JQ-05), Shaanxi Fundamental Science Research Project for Mathematics and Physics (No. 23JSZ010) and Fundamental Research Funds for the Central Universities (No. 20199235177). X. Yu is supported by the Hong Kong RGC General Research Fund (GRF) under grant no. 15304122 and grant no. 15306523.}


\begin{thebibliography}{}
{\small

\bibitem[Abraham(2000)]{Abraham2000} Abraham, R. (2000): Reflecting Brownian snake and a Neumann-Dirichlet problem. \emph{Stoch. Process. Appl.} 89(2), 239-260.

\bibitem[Brosamler(1976)]{Brosamler1976} Brosamler, G.A. (1976): A probabilistic solution of the Neumann problem. \emph{Math. Scand.} 38, 137-147.

\bibitem[Dupuis and Ishii(1991)]{DupuisIshii1991} Dupuis, P., H. Ishii (1991): On oblique derivative problems for fully nonlinear second-order elliptic PDEs on domains with corners. \emph{Hokkaido Math. J.} 20, 135-164.

\bibitem[Freidlin(1985)]{Freidlin1985} Freidlin, M.I. (1985): \emph{Functional Integration and Partial Differential Equations}. Princeton University Press, Princeton, NJ.


\bibitem[Harrison (1985)]{Harrison85} Harrison, J.M. (1985): {\it  Brownian Motion and Stochastic Flow Systems.} Wiley, New York.

\bibitem[Harrison and Reiman(1981)]{Harrison81} Harrison, J.M., M.I. Reiman (1981): Reflected Brownian motion on an orthant. \emph{Ann. Probab.} 9(2), 302-308.

\bibitem[Hsu(1984)]{Hsu1984} Hsu, P. (1984): \emph{Reflecting Brownian Motion, Boundary Local Time and the Neumann Problem}. Ph.D Dissertations, Stanford University.

\bibitem[Hsu(1985)]{Hsu1985} Hsu, P. (1985): Probabilistic approach to the Neumann problem. \emph{Comm. Pure Appl. Math.} 38(4), 444-472.

\bibitem[Hu(1993)]{Hu1993} Hu, Y. (1993): Probabilistic interpretation of a system of quasilinear elliptic partial differential equations under Neumann boundary
conditions. \emph{Stoch. Process. Appl.} 48, 107-121.

\bibitem[Ishii and Kumagai(2022)]{IsKu2022} Ishii, H., T. Kumagai (2022): Nonlinear neumann problems for fully nonlinear elliptic pdes on a quadrant.
\emph{SIAM J. Math. Anal.}  54(6), 5854-5887.

\bibitem[Leimkuhler et al.(2023)]{LeiShatre23} Leimkuhler, B., A. Sharma, M.V. Tretyakov (2023): Simplest random walk for approximating Robin boundary value problems and ergodic limits of reflected diffusions. \emph{Ann. Appl. Probab.} 33(3), 1904-1960.

\bibitem[Ma and Song(1988)]{masong1988} Ma, Z.M., R.M. Song (1988): Probabilistic approach to the Neumann problem with infinite gauge. \emph{Acta Math. Appl. Sinica} 4, 30-40.


\bibitem[Milstein and Tretyakov(2002)]{MilTre2002} Milstein, G.N., M.V. Tretyakov (2002): A probabilistic approach to the solution of the Neumann problem for nonlinear parabolic equations. \emph{IMA J. Numer. Anal.} 22, 599-622.

\bibitem[Papanicolaou(1990)]{Papanicolaou1990} Papanicolaou, V.G. (1990):  The probabilistic solution of the third boundary value problem for second order elliptic equations.  \emph{Probab. Theor. Rel. Fields.} 87(1), 27-77.


\bibitem[Veestraeten(2004)]{Vee2004} Veestraeten, D. (2004): The conditional probability density function for a reflected Brownian motion. {\it Comput. Econom.} {24}(2), 185-207.

\bibitem[Wong et al.(2022)]{wongetal2022} Wong, C.H., X. Yang, J. Zhang (2022): A probabilistic approach to Neumann problems for elliptic PDEs with nonlinear divergence terms. \emph{Stoch. Process. Appl.} 151, 101-126.

\bibitem[Zhang(1990)]{zhang1990} Zhang, T.S. (1990): Probabilistic approach to the Neumann problem. \emph{Acta Math. Appl. Sinica} 6, 126-134.

\bibitem[Zhou et al.(2017)]{zhouetal2017} Zhou, Y.J., W. Cai, E. Hsu (2017): Computation of the local time of reflecting Brownian motion and the probabilistic representation of the Neumann problem. \emph{Commun. Math. Sci.} 15(1), 237-259.

}
\end{thebibliography}
\end{document}